\documentclass[11pt]{article}

\usepackage{paralist}
\usepackage[margin=3.0cm]{geometry} 
\usepackage{graphicx} 

\usepackage{amssymb, amsfonts, amsmath, amscd, amsthm} 
\usepackage{mathtools} 
\usepackage{titlesec} 
\usepackage{csquotes} 
\usepackage[shortlabels]{enumitem} 
\usepackage{hyperref} 
\usepackage{mathrsfs} 
\usepackage{bbm}
\usepackage{multicol}
\usepackage[misc]{ifsym}
\usepackage{comment}

\hypersetup{
     colorlinks   = true,
     citecolor    = green
}

\numberwithin{equation}{section}

\titleformat*{\section}{\Large \scshape\center}
\titleformat*{\subsection}{\fontsize{14}{14} \sffamily}

\theoremstyle{plain}
\newtheorem{theorem}{Theorem}[section]
\newtheorem*{theorem*}{Theorem}

\newtheorem{lemma}[theorem]{Lemma}
\newtheorem{proposition}[theorem]{Proposition}
\newtheorem{thm*}{Theorem}
\newtheorem{prop*}{Proposition}
\newtheorem{corollary}[theorem]{Corollary}

\theoremstyle{definition}
\newtheorem{definition}[theorem]{Definition}
\newtheorem*{definition*}{Definition}

\newtheorem{assumption}{Assumption}

\theoremstyle{remark}
\newtheorem{remark}[theorem]{Remark}

\newcommand{\BUC}{\operatorname{BUC}}

\DeclareMathOperator{\tr}{tr}

\allowdisplaybreaks

\newcommand{\C}{\mathbb{C}}

\newcommand{\N}{\mathbb{N}}
\newcommand{\R}{\mathbb{R}}
\newcommand{\Z}{\mathbb{Z}}
\newcommand{\1}{\mathbbm{1}}

\newcommand{\Cc}{\mathcal{C}}

\newcommand{\Hc}{\mathcal{H}}
\newcommand{\Jc}{\mathcal{J}}
\newcommand{\Kc}{\mathcal{K}}
\newcommand{\Lc}{\mathcal{L}}

\newcommand{\Wc}{\mathcal{W}}

\DeclareMathOperator{\BO}{BO}
\DeclareMathOperator{\BDO}{BDO}

\DeclareMathOperator{\supp}{supp}
\DeclareMathOperator{\Co}{Co}
\DeclareMathOperator{\spec}{spec}
\DeclareMathOperator{\ess}{ess}

\newcommand{\from}{\colon}

\newcommand{\scpr}[2]{\left\langle #1, #2 \right\rangle}
\renewcommand{\sp}{\scpr}
\newcommand{\abs}[1]{\left\lvert#1\right\rvert}
\newcommand{\norm}[1]{\left\lVert#1\right\rVert}
\newcommand{\set}[1]{\left\{ #1\right\}}

\newcommand{\vertiii}[1]{{\left\vert\kern-0.25ex\left\vert\kern-0.25ex\left\vert #1
    \right\vert\kern-0.25ex\right\vert\kern-0.25ex\right\vert}}

\renewcommand{\epsilon}{\varepsilon}

\begin{document}
\pagenumbering{gobble}
\title{Fredholm operators on abelian phase spaces}
\author{Robert Fulsche and Raffael Hagger}
\date{}
\maketitle
\pagenumbering{arabic}

\begin{abstract}
We study the Fredholm property for linear operators on coorbit spaces over locally compact abelian phase spaces. As a special case we consider operators on $L^2(G)$, where $G$ is an arbitrary locally compact abelian group. Our approach therefore extends the existing theory for discrete spaces to the continuous setting and complements the study in \cite{Fulsche_Hagger2025}, where compactness was characterized in terms of limit operators. Our results are again achieved by merging tools from the theory of band-dominated operators with methods of quantum harmonic analysis, thereby achieving new results in both areas. We emphasize that our results are new (and maybe most interesting) even in the $L^2$-setting.
\end{abstract}



\section{Introduction}

Operator theory on function spaces over topological groups is a topic of everlasting interest as it reveals connections between the analytical, geometrical and algebraical properties of the underlying structures. Some of the first operators to study are the compact operators. As a rule of thumb, a compact operator vanishes at the boundary of the domain in a sense that can be made precise with the help of limit operators. The converse is only true if one assumes the membership in an appropriate algebra as well as some geometric properties of the domain close to the boundary \cite{Fulsche_Hagger2025}. The next step is to study Fredholm operators, that is, those operators that are invertible modulo compact operators. The expected result is that Fredholm operators are invertible at the boundary, in analogy to what happens for compact operators. This is true most of the time, but usually turns out to be much harder to prove.

Let us briefly recall a few results in this direction. The most classical example is the $\ell^p$-sequence space over $\Z$. In this case every operator can be depicted as an infinite matrix and one can show quite easily that an operator is compact if and only if it is band-dominated (that is, can be approximated by band matrices) and its diagonals vanish at infinity. The corresponding characterization of Fredholm operators reads as follows: A band-dominated operator on $\ell^p(\Z)$ is Fredholm if and only if all of its limit operators are invertible. Again, these limit operators are a way to describe how an operator behaves at infinity. The latter result has been shown by Lindner and Seidel \cite{Lindner_Seidel} only in 2014, hinting at the fact that this problem is much harder than the compactness result. Only three years later, this characterization of Fredholm operators was generalized to discrete metric spaces with certain properties involving the growth towards its boundary, revealing the geometric nature of the result \cite{Spakula_Willett}.  More recently, the theory of limit operators has also been expanded in a slightly different direction, for example to Bergman \cite{Hagger2017,Hagger2019} and Fock spaces \cite{Fulsche_Hagger2019}. Again, the algebra of band-dominated operators appears quite naturally in the form of the Toeplitz algebra \cite{hagger2021}. As the latest development in this area, it was demonstrated in \cite{HaggerSeifert} that this toolkit can be used in a rather general setting of metric measure spaces, generalizing most of the previously known results. 

One could think that this would be the end of the story. In many situations, such as locally compact abelian (lca) groups, the tools from \cite{HaggerSeifert} apply to obtain characterizations of compactness and Fredholm properties for linear operators on function spaces over the underlying group. Unfortunately, there is a major flaw in the picture here: The method in \cite{HaggerSeifert} assumes that certain projections are locally compact. If the underlying space is not discrete, this restricts the applicability to ``nice'' subspaces such as the Bergman and Fock spaces mentioned above. It does not, however, characterize the Fredholm operators on $L^2(\R)$. Moreover, most of the previous work assumes that the underlying space is properly metrizable, ruling out any lca group which is not second countable. In particular, this excludes uncountable discrete spaces.

The main problem on non-discrete spaces is that the behavior at the boundary is not sufficient to decide whether an operator is Fredholm. For instance consider any non-zero function $f \from \R \to \R$ with compact support. Then the multiplication operator $M_{1-f}$ is not Fredholm. However, at infinity, the operator looks just like the identity operator. One therefore has to either restrict the algebra in which we characterize Fredholm operators or reinterpret what we mean by limit operators. In the Hilbert spaces case the former approach was pursued by Georgescu and Iftimovici \cite{Georgescu2011,GeorgescuIftimovici2006}, for instance. They considered the \emph{elliptic algebra}, which is defined as the norm closure of all band operators that are also integral operators with uniformly continuous kernel. In particular, it was shown in \cite[Theorem 1.1]{GeorgescuIftimovici2006} (see also \cite[Theorem 2.5]{Georgescu2011}) that an operator in the elliptic algebra is Fredholm if and only if all limit operators are invertible and their inverses are uniformly bounded. Note that the elliptic algebra is much smaller than the algebra of band-dominated operators for non-discrete spaces (cf. \cite[Section 7]{Georgescu2011}). It is also smaller than the algebra considered in \cite{Fulsche_Hagger2025}, where a characterization of compact operators that resembles the one for sequence spaces was obtained. This algebra, which we will call $\Cc_1^p$ below, appears very naturally in the context of quantum harmonic analysis and will again be our main object of study. Our approach is to combine the methods of band-dominated operators discussed in \cite{HaggerSeifert} with tools from \emph{quantum harmonic analysis}, which themselves have been proven to be very useful in operator theory in the last few years; see \cite{Dewage_2026,Fulsche2020,Fulsche_Galke2023,Fulsche_Hagger2024,Fulsche_Luef_Werner2024,Halvdansson2023,Luef_Skrettingland2021} just to name a few recent publications.

Let us now turn to the scope of the present work. Our setting is that of an \emph{abelian phase space}, by which we mean a locally compact abelian group $\Xi$ endowed with a multiplier $m$ that induces a certain symplectic self-duality of the group $\Xi$ with its Pontryagin dual $\widehat{\Xi}$. This setup is described more precisely in the main body of the text. The leading example of such a phase space is the situation where $G$ is any locally compact abelian group, $\Xi = G \times \widehat{G}$ and $m((x, \xi), (y, \eta)) = \overline{\xi(y)}$. Nevertheless, it has been shown that there are indeed other more exotic examples of groups that fall into this framework \cite{Prasad_Shapiro_Vemuri2010}. Within this setting of an abelian phase space, the results of Baggett and Kleppner \cite{Baggett_Kleppner1973} (extending the classical Stone-von Neumann theorem) show that there exists a unique irreducible projective representation $U$ of $\Xi$ with multiplier $m$. The spaces we consider are then exactly the $p$-Coorbit spaces $\Co_p(U)$ of this representation. For simplicity, we restrict ourselves to the reflexive situation $1 < p < \infty$ in this paper. In the standard setting $\Xi = G \times \widehat{G}$, these Coorbit spaces precisely agree with the modulation spaces $M^p(G)$. In particular, we get $\Co_p(U) = L^2(G)$ for $p = 2$.

A key object in our studies is the algebra $\mathcal C_1^p$, consisting of bounded operators $A$ on the coorbit space for which the map $\Xi \ni x \mapsto U_x A U_x^{-1} =: \alpha_x(A)$ is continuous in operator norm topology. In this case the map $x \to \alpha_x(A)$ can be extended to a strongly continuous map on a certain compactification $\sigma\Xi$ of $\Xi$. Its boundary $\sigma\Xi \setminus \Xi$ will be denoted by $\partial\Xi$. With these notations we can now spell out our main result.

\begin{theorem*} Let $(\Xi, m)$ be an abelian phase space, $1 < p < \infty$ and $A \in \Cc_1^p$. Then $A$ is Fredholm if and only if $\alpha_x(A)$ is invertible for every $x \in \partial\Xi$. In particular, the essential spectrum of $A$ is given by
\[\spec_{\ess}(A) = \bigcup\limits_{x \in \partial\Xi} \spec(\alpha_x(A)).\]
\end{theorem*}

The paper is structured as follows: In Section \ref{sec:coorbit} we recall some basic notation, i.e., the details on the locally compact abelian group $\Xi$ and the multiplier $m$, together with the respective notions on the Coorbit spaces. Section \ref{sec:limit_ops} will explain the notion of limit operators, yielding a rigorous notion of the behavior of the map $x \mapsto \alpha_x(A)$ at infinity. In Section \ref{sec:band_dominated} we introduce and discuss our notion of band-dominated operators on abelian phase spaces. This continues into Section \ref{sec:norm_localization}, where the operator norm localization property of band-dominated operators is discussed. Finally, in Section \ref{sec:Fredholm}, we combine the tools from band-dominated operator theory with the methods from quantum harmonic analysis to obtain our results on the Fredholm property.

\section{Coorbit spaces}\label{sec:coorbit}

The general setup of the present work is that of coorbit spaces over abelian phase spaces. The theory of coorbit spaces was initially developed in the 80s by H.~G.~Feichtinger and K.~Gröchenig, cf.\ \cite{Feichtinger_Groechenig1989a, Feichtinger_Groechenig1989b}; see also \cite{Christensen1996} for a detailed discussion in the context of projective unitary representations. For this paper, we will use the notation and conventions introduced in \cite[Section 2]{Fulsche_Hagger2025} and add a few simple results in this section. Let us first recall the basic setting, which is identical to the setting presented in Assumptions 1 and 2 of \cite[Section 2]{Fulsche_Hagger2025}.

\begin{assumption} \label{ass_1}
    We assume that $(\Xi, m)$ is an abelian phase space, i.e., $\Xi$ is a locally compact abelian (lca) group with identity $e$ and $m \from \Xi \times \Xi \to S^1$ is a separately continuous 2-cocycle on $\Xi$ that satisfies
    \[m(xy,z)m(x,y) = m(x,yz)m(y,z)\]
    for $x,y,z \in \Xi$, Moreover, we assume that $x \mapsto \sigma(x,\cdot) := m(x,\cdot)/m(\cdot,x)$ is a topological group isomorphism from $\Xi$ to $\widehat{\Xi}$, $m(x,y) = m(x^{-1}, y^{-1})$ for all $x, y\in \Xi$ and $m(e,e) = 1$.
\end{assumption}

Note that the latter assumption implies $m(e,x) = 1$ for all $x \in \Xi$. Let $(\mathcal H, (U_x)_{x \in \Xi})$ be the unique (up to unitary equivalence) irreducible projective unitary representation with $m$ as a cocycle, which means that we have
\[U_xU_y = m(x,y)U_{xy}\]
for $x,y \in \Xi$. The assumption $m(x,y) = m(x^{-1},y^{-1})$ for $x,y \in \Xi$ ensures that there exists a unitary operator $R \in \mathcal U(\mathcal H)$ satisfying $RU_x = U_{x^{-1}}R$. $R$ can be chosen to be self-adjoint (cf.\ \cite[Section 2]{Fulsche_Galke2023}). We also note that $U_x^\ast = \overline{m(x,x^{-1})}U_{x^{-1}}$ for all $x \in \Xi$.
    
\begin{assumption} \label{ass_2}
    We assume that $(U_x)_{x \in \Xi}$ is integrable. That is, the set
    \[\mathcal H_1 := \{ f \in \mathcal H: ~x \mapsto \langle f, U_x f\rangle \in L^1(\Xi)\}\]
    is assumed to be nontrivial.
\end{assumption}

Here, $L^1(\Xi)$ denotes the usual Lebesgue space of equivalence classes of integrable functions with respect to the Haar measure $\lambda$ on $\Xi$. In integrals we will frequently write $\mathrm{d}x$ instead of $\mathrm{d}\lambda(x)$.

The set $\Hc_1$ turns out to be a dense subspace of $\Hc$. Fix $\varphi_0 \in \mathcal H_1$ with $\norm{\varphi_0}_{\Hc} = 1$. For $f \in \mathcal H$ and $x \in \Xi$ we will write
\[\mathcal W_{\varphi_0}(f)(x) := \langle f, U_x \varphi_0\rangle\]
for the \emph{Wavelet transform} of $f$ with window $\varphi_0$. It is well known that $\Wc_{\varphi_0} \from \Hc \to L^2(\Xi)$ is an isometry that can be extended to the coorbit spaces $\Co_p(U)$, $p \in [1,\infty]$, which are then endowed with the norm
\[\| f\|_{p, \varphi_0} := \| \mathcal W_{\varphi_0}(f)\|_{L^p(\Xi)}.\]
We will abbreviate $W_{\varphi_0}^p := \mathcal W_{\varphi_0}(\Co_p(U)) \subseteq L^p(\Xi)$ and recall that this is a closed subspace of $L^p(\Xi)$ consisting of bounded continuous functions. The orthogonal projection $P_{\varphi_0} := \Wc_{\varphi_0}\Wc_{\varphi_0}^*$ onto $W_{\varphi_0}^2$ is given by the twisted convolution
\[P_{\varphi_0}f = f \ast' \mathcal W_{\varphi_0}(\varphi_0),\]
where
\[(f \ast' g)(x) = \int_\Xi f(y) g(y^{-1}x)m(y,y^{-1}x) \, \mathrm{d}y.\]

\begin{lemma}\label{lemma:projection_wiener_algebra}
    There exists an $H \in L^1(\Xi)$ such that the integral kernel $k(x,y)$ of $P_{\varphi_0}$ satisfies
    \begin{align*}
        |k(x,y)| \leq H(y^{-1}x)
    \end{align*}
    for all $x,y \in \Xi$. In particular, $P_{\varphi_0}$ acts as a continuous projection from $L^p(\Xi)$ to $W_{\varphi_0}^p$ for every $p \in [1,\infty]$.
\end{lemma}

\begin{proof}
As the integral kernel of $P_{\varphi_0}$ is given by
\[k(x,y) = \mathcal W_{\varphi_0}\varphi_0(y^{-1}x) m(y,y^{-1}x)\]
and $\Wc_{\varphi_0}\varphi_0 \in L^1(\Xi)$, the first part of the lemma is clear and the second part follows from Young's inequality.
\end{proof}

Define
\[V_x f(y) := \frac{m(x,x^{-1})}{m(x^{-1},y)}f(x^{-1}y)\]
for functions $f \from \Xi \to \mathbb C$, $x,y \in \Xi$. The Wavelet transform intertwines between $V_x$ and $U_x$, that is, $\Wc_{\varphi_0}U_x = V_x\Wc_{\varphi_0}$ for all $x \in \Xi$. We note that $(V_x)_{x \in \Xi}$ is a projective unitary representation on $L^2(\Xi)$ that is continuous in the strong operator topology and satisfies $V_xV_y = m(x,y)V_{xy}$ for all $x,y \in \Xi$. Moreover, $V_x \from L^p(\Xi) \to L^p(\Xi)$ is a surjective isometry for every $p \in [1,\infty]$ and
\[V_x^{-1} = \overline{m(x,x^{-1})}V_{x^{-1}}.\]

\begin{corollary} \label{cor:reproducing_formula}
    For $g \in W_{\varphi_0}^p$ and $x \in \Xi$ we have the reproducing formula
    \[g(x) = \sp{g}{\Wc_{\varphi_0}U_x\varphi_0} = \sp{g}{V_x\Wc_{\varphi_0}\varphi_0}.\]
\end{corollary}

Next, we show that $P_{\varphi_0}$ commutes with the representation $(V_x)_{x \in \Xi}$.

\begin{lemma} \label{lem:commutation_properties_V_z}
    For $x \in \Xi$ we have $V_x P_{\varphi_0} = P_{\varphi_0}V_x$. Moreover, $V_xM_f = M_{f(x^{-1}\cdot)}V_x$ for all $f \in L^{\infty}(\Xi)$.
\end{lemma}

\begin{proof}
For $p = 2$ we have $P_{\varphi_0} = \mathcal W_{\varphi_0} \mathcal W_{\varphi_0}^\ast$ and hence
\begin{align*}
    V_x P_{\varphi_0} = V_x \mathcal W_{\varphi_0} \mathcal W_{\varphi_0}^\ast = \mathcal W_{\varphi_0} U_x  \mathcal W_{\varphi_0}^\ast = \mathcal W_{\varphi_0}\mathcal W_{\varphi_0}^\ast  V_x.
\end{align*}
For $1 \leq p < \infty$ the equality now follows by density. For $p = \infty$ the statement follows by duality.

The second statement is immediate:
\[V_zM_fg(x) = \frac{m(z,z^{-1})}{m(z^{-1},x)}f(z^{-1}x)g(z^{-1}x) = M_{f(z^{-1}\cdot)}V_zg(x).\qedhere\]

\end{proof}

We define Toeplitz operators on $W_{\varphi_0}^p$ as follows: Given a symbol $f \in L^\infty(\Xi)$, the operator $T_{f, \varphi_0}$ acts on $W_{\varphi_0}^p$ as $T_{f, \varphi_0}(g) := P_{\varphi_0}(fg)$. On the other hand, we can also write a Toeplitz operator $T_{f, \varphi_0}$ as a convolution of the symbol $f$ with a rank one operator. Recall that the convolution between a nuclear operator $A \from W_{\varphi_0}^p \to W_{\varphi_0}^p$ and $f \in L^1(\Xi)$ is defined by the Bochner integral
\[f \ast A := A \ast f := \int_{\Xi} f(x)\alpha_x(A) \, \mathrm{d}x,\]
where $\alpha_x(A) := V_xAV_x^{-1}$. This can be extended to the case where $f \in L^\infty(\Xi)$ and $A \in W_{\varphi_0}^1 \hat{\otimes}_{\pi} W_{\varphi_0}^1$, at least when $1 < p < \infty$, which we assume in the following. Further, for $A \in \mathcal N(W_{\varphi_0}^p)$ (the nuclear operators on $W_{\varphi_0}^p$) and $B \in \mathcal L(W_{\varphi_0}^p)$ (with $1 < p < \infty$) we can define
\begin{align*}
    A \ast B(x) = \operatorname{tr}(A \alpha_x \beta_-(B)), \quad x \in \Xi.
\end{align*}
Here, $\operatorname{tr}$ denotes the nuclear trace and we also use the notation $\beta_-(B) = RBR$ (with $R$ being the parity operator). When $A$ is nuclear and $B$ is bounded, then $A \ast B$ is a bounded and continuous function on $\Xi$. When both $A$ and $B$ are contained in $W_{\varphi_0}^1 \hat{\otimes}_{\pi} W_{\varphi_0}^1$, then $A \ast B$ yields a function in $L^1(\Xi)$. 
For details we refer to \cite[Section 7]{Fulsche_Galke2023}.

\begin{lemma}
For $f \in L^\infty(\Xi)$ and $p \in (1,\infty)$ we have $T_{f, \varphi_0} = (\mathcal W_{\varphi_0}\varphi_0 \otimes \mathcal W_{\varphi_0} \varphi_0) \ast f$.    
\end{lemma}

\begin{proof}
    Let $g \in W_{\varphi_0}^p$. Then, using the cocycle identity
    \[m(y,y^{-1}x)m(y^{-1},x) = m(e,x)m(y,y^{-1}) = m(y,y^{-1}),\]
    we obtain
    \begin{align*}
        T_{f, \varphi_0}(g)(x) &= P_{\varphi_0}(f g)(x)\\
        &= \left((fg) \ast' \mathcal W_{\varphi_0} \varphi_0\right)(x)\\
        &= \int_\Xi f(y) g(y) \mathcal W_{\varphi_0}\varphi_0(y^{-1}x) m(y,y^{-1}x)~\mathrm{d}y\\
        &= \int_\Xi f(y) g(y) \frac{m(y^{-1},x)}{m(y,y^{-1})}\langle U_y \varphi_0, U_x \varphi_0\rangle m(y,y^{-1}x)~\mathrm{d}y\\
        &= \int_\Xi f(y) g(y) \langle U_y \varphi_0, U_x \varphi_0\rangle~\mathrm{d}y\\
        &= \int_\Xi f(y) g(y) \mathcal W_{\varphi_0} (U_y \varphi_0)(x)~\mathrm{d}y.
    \end{align*}
    On the other hand (in an appropriate weak sense),
    \begin{align*}
        \left(f \ast (\mathcal W_{\varphi_0}\varphi_0 \otimes \mathcal W_{\varphi_0}\varphi_0)\right)(g) &= \int_\Xi f(y) \langle g, V_y \mathcal W_{\varphi_0}\varphi_0\rangle V_y \mathcal W_{\varphi_0}\varphi_0~\mathrm{d}y\\
        &= \int_\Xi f(y) g(y) V_y \mathcal W_{\varphi_0}\varphi_0 ~\mathrm{d}y
    \end{align*}
    by Corollary \ref{cor:reproducing_formula}. Since the Wavelet transform intertwines $V_y$ and $U_y$, equality follows.
\end{proof}

We will also make use of the following facts.

\begin{lemma}
    Let $K \subseteq \Xi$ be compact and $p \in (1,\infty)$. Then, $P_{\varphi_0} M_{\mathbf 1_K}: W_{\varphi_0}^p \to W_{\varphi_0}^p$ and $(I - P_{\varphi_0})M_{\mathbf 1_K}: W_{\varphi_0}^p \to L^p(\Xi)$ are compact.
\end{lemma}

\begin{proof}
    Note that $P_{\varphi_0}M_{\mathbf 1_K} = T_{\mathbf 1_K, \varphi_0} = (\mathcal W_{\varphi_0}\varphi_0 \otimes \mathcal W_{\varphi_0}\varphi_0) \ast \mathbf 1_K$. Since $\mathcal W_{\varphi_0}\varphi_0 \otimes \mathcal W_{\varphi_0}\varphi_0$ acts as a nuclear operator on $W_{\varphi_0}^p$ and $\mathbf 1_K \in L^1(\Xi)$, this convolution even yields a nuclear operator on $W_{\varphi_0}^p$; cf.\ \cite[Section 7]{Fulsche_Galke2023}.

    For $p = 2$ the compactness of $(I - P_{\varphi_0})M_{\mathbf 1_K}$ can be seen as follows. Define the Hankel operator 
    \[H_{\mathbf 1_K, \varphi_0} := (I - P_{\varphi_0})M_{\mathbf 1_K} \from W_{\varphi_0}^2 \to L^2(\Xi).\]
    It follows
    \begin{align*}
    H_{\mathbf 1_K, \varphi_0}^\ast H_{\mathbf 1_K, \varphi_0} = T_{\mathbf 1_K, \varphi_0} - T_{\mathbf 1_K, \varphi_0}^2
    \end{align*}
    and since $T_{\mathbf 1_K, \varphi_0}$ is compact, $H_{\mathbf 1_K, \varphi_0}$ is compact as well.

    We are rather certain that one can prove the compactness of $H_{\mathbf 1_K, \varphi_0}$ directly by considering it as an integral operator on $L^p(\Xi)$ and employing results on the $pq$-norm of the kernel (see, e.g., \cite[Theorem 41.6]{Zaanen1997}). However, such theorems are usually formulated for separable measure spaces. Since we do not want to spend a significant part of the paper on discussions about such theorems for non-separable measure spaces, we mention that the compactness of $H_{\mathbf 1_K, \varphi_0}$ in case $p = 2$ implies that $H_{\mathbf 1_K, \varphi_0}: W_{\varphi_0}^p \to L^p(\Xi)$ is compact for each $1 < p < \infty$ by results on the complex interpolation of compact operators, see \cite{Cwikel1992, Persson1963}, for instance.
\end{proof}

\begin{corollary} \label{cor:multiplication_compact}
    Let $K \subseteq \Xi$ be compact and $p \in (1,\infty)$. Then both $P_{\varphi_0}M_{\mathbf 1_K} \from L^p(\Xi) \to L^p(\Xi)$ and $M_{\mathbf 1_K}P_{\varphi_0}: L^p(\Xi) \to L^p(\Xi)$ are compact. 
\end{corollary}

\begin{proof}
    By the previous lemma, we obtain that $M_{\mathbf 1_K}: W_{\varphi_0}^p \to L^p(\Xi)$ is compact, hence also $M_{\mathbf 1_K}P_{\varphi_0}$. The compactness of $P_{\varphi_0}M_{\mathbf 1_K}$ follows by taking adjoints.
\end{proof}

\section{Properties of limit operators}\label{sec:limit_ops}

We recall the concept of limit operators and give some of its properties. The facts presented here are straightforward generalizations of the Hilbert space case presented in \cite[Section 3.1]{Fulsche_Luef_Werner2024}. Nevertheless, for the reader's convenience, we decided to add some details on this. Throughout this section, we will assume that $p \in (1,\infty)$.

Let $\mathcal N(W_{\varphi_0}^p)$ denote the ideal of nuclear operators in $\Lc(W_{\varphi_0}^p)$. The strong continuity of $x \mapsto V_x$ implies the norm continuity of
\[x \mapsto \alpha_x(A) := V_xAV_x^{-1}\]
for any finite rank operator $A$ and therefore, by approximation, for every $A \in \mathcal N(W_{\varphi_0}^p)$. As $W_{\varphi_0}^p$ is reflexive, the dual of $\mathcal N(W_{\varphi_0}^p)$ can be identified with $\mathcal L(W_{\varphi_0}^p)$ through trace duality: For $A \in \mathcal L(W_{\varphi_0}^p)$, the corresponding functional is given by $\varphi_A(N) = \tr(NA)$, and every bounded linear functional on $\mathcal N(W_{\varphi_0}^p)$ is of this form.

Given $A \in \mathcal L(W_{\varphi_0}^p)$ and $N \in \mathcal N(W_{\varphi_0}^p)$, the map
\begin{align*}
    \Xi \ni x \mapsto f_{N,A}(x) := \tr(N \alpha_x(A))
\end{align*}
is a bounded uniformly continuous function on $\Xi$. Hence, it extends to a continuous function on $\sigma \Xi := \mathcal M(\BUC(\Xi))$, which we again denote by $f_{N,A}$. Here, $\mathcal M(\BUC(\Xi))$ denotes the maximal ideal space of the bounded uniformly continuous functions on $\Xi$, which we consider as a compactification of $\Xi$ by identifying points in $\Xi$ with the corresponding point evaluation functionals. In addition, we will write $\partial \Xi := \sigma \Xi \setminus \Xi$ for the boundary of $\Xi$ in this compactification. Then, for each $x \in \partial \Xi$, we still have $|f_{N, A}(x)| \leq \| N\|_{\mathcal N} \| A\|_{op}$. Clearly, for fixed $x \in \sigma\Xi$, $f_{N, A}(x)$ is linear in $N$, and therefore acts as a bounded linear functional on $\mathcal N(W_{\varphi_0}^p)$. Therefore, there exists a unique $\alpha_x(A) \in \mathcal L(W_{\varphi_0}^p)$ such that for all $N \in \mathcal N(W_{\varphi_0}^p)$ the equaility $f_{N, A}(x) = \tr(N\alpha_x(A))$ holds. In summary, we get the following lemma.

\begin{lemma}
    Let $A \in \mathcal L(W_{\varphi_0}^p)$. Then, the map $\Xi \ni x \mapsto \alpha_x(A)$ uniquely extends to a weak$^\ast$-continuous map from $\sigma \Xi$ to $\mathcal L(W_{\varphi_0}^p)$.
\end{lemma}

The operators $\alpha_x(A)$, $x \in \partial \Xi$, are called \emph{limit operators of $A$}. Here are two important properties of limit operators:

\begin{lemma} \label{lem:limit_operator_properties}
    Let $A, B \in \mathcal L(W_{\varphi_0}^p)$ and $x \in \sigma\Xi$. 
    \begin{enumerate}
        \item[(i)] For $\lambda, \mu \in \mathbb C$ we have $\alpha_x(\lambda A + \mu B) = \lambda \alpha_x(A) + \mu \alpha_x(B)$.
        \item[(ii)] The inequality $\| \alpha_x(A)\|_{op} \leq \| A\|_{op}$ holds.
    \end{enumerate}
\end{lemma}

\begin{proof}
    Both properties follow from the properties of the function $f_{N,A}$ described above.
\end{proof}
In the following, we will frequently work with operators from the Banach algebra
\begin{align*}
    \mathcal C_1^p(\varphi_0) := \{A \in \mathcal L(W_{\varphi_0}^p): ~\| A - \alpha_x(A)\|_{op} \to 0 \text{ as } x \to e\}.
\end{align*}
If $A \in \mathcal C_1^p(\varphi_0)$, the map $x \mapsto \alpha_x(A)$ extends to a strongly continuous map on $\sigma\Xi$. To see this, we briefly need to discuss the classical analogue of limit operators, namely limit functions. For any function $f \from \Xi \to \C$ and $x \in \Xi$ define
\[\alpha_x(f) \from \Xi \to \C, \quad \alpha_x(f)(y) := f(x^{-1}y).\]
By a similar argument to that for the duality $\mathcal N(W_{\varphi_0}^p)' \cong \mathcal L(W_{\varphi_0}^p)$ described above, the following holds:

\begin{lemma}[{\cite[Section 2.1]{Fulsche_Luef_Werner2024}}]
    Let $f \in L^\infty(\Xi)$. Then, the map
    \begin{align*}
        \Xi \ni x \mapsto \alpha_x(f) \in L^\infty(\Xi)
    \end{align*}
    uniquely extends to a weak$^\ast$-continuous map from $\sigma \Xi$ to $L^\infty(\Xi)$. The functions $\alpha_x(f)$, $x \in \partial \Xi$, are called \emph{limit functions} and satisfy
    \begin{align*}
        \alpha_x(\lambda f + \mu g) = \lambda \alpha_x(f) + \mu \alpha_x(g), \quad \| \alpha_x(f)\|_\infty \leq \| f\|_\infty, \quad \text{for } f, g \in L^\infty(\Xi), x \in \partial \Xi, \lambda,\mu \in \C.
    \end{align*}
\end{lemma}

For $f \in \BUC(\Xi)$, the theorem of Arzel\'{a}-Ascoli implies that the map $\sigma \Xi \ni x \mapsto \alpha_x(f)$ is continuous with respect to the compact-open topology; cf.\ \cite[Proposition 2.8]{Fulsche_Luef_Werner2024}. With this fact available, we now return to limit operators.

We first mention that $\Co_p(U)$ has the bounded approximation property (BAP) for every $1 < p < \infty$. Indeed, as already mentioned earlier, $\Co_p(U)$ is isometrically equivalent to $W_{\varphi_0}^p$, and $W_{\varphi_0}^p$ is a closed subspace of $L^p(\Xi)$. Since each $L^p(\Xi)$ has the bounded approximation property and $P_{\varphi_0}$ is a bounded projection from $L^p(\Xi)$ onto $W_{\varphi_0}^p$, it follows that $W_{\varphi_0}^p$ has the bounded approximation property, hence so does $\Co_p(U)$. This implies, in particular, that the class of nuclear operators on $\Co_p(U)$ can be written as a projective tensor product:
\begin{align*}
    \mathcal N(\Co_p(U)) \cong \Co_p(U) \widehat{\otimes}_\pi \Co_p(U)' \cong \Co_p(U) \widehat{\otimes}_\pi Co_{q}(U), \quad \text{ where } \frac{1}{p} + \frac{1}{q} = 1.
\end{align*}
Since $\Co_1(U)$ is dense in both $\Co_p(U)$ and $\Co_q(U)$, we see that $\Co_1(U) \widehat{\otimes}_\pi \Co_1(U)$ is dense in $\mathcal N(\Co_p(U))$. We write $\mathcal N_0 = W_{\varphi_0}^1 \hat{\otimes}_\pi W_{\varphi_0}^1$. By what was just mentioned, $\mathcal N_0$ is densely contained in $\mathcal N(W_{\varphi_0}^p)$ for each $1 < p < \infty$: We have the following important fact:

\begin{proposition}\label{prop:C1N0}
    For each $1 < p < \infty$ it is $\overline{\mathcal N_0 \ast \BUC(\Xi)} = \mathcal C_1^p(\varphi_0)$. 
\end{proposition}

\begin{proof}
    By basic properties of convolutions, we see that $\mathcal N_0 \ast \BUC(\Xi) \subseteq \mathcal C_1^p(\varphi_0)$, as well as $\mathcal N_0 \ast \mathcal C_1^p(\varphi_0) \subseteq \BUC(\Xi)$. The latter of course also yields $\mathcal N_0 \ast \mathcal N_0 \ast \mathcal C_1^p(\varphi_0) \subseteq \mathcal C_1^p(\varphi_0)$. Since $\mathcal N_0 \ast \mathcal N_0 \subseteq L^1(\Xi)$ densely, we see that (using that $\mathcal C_1^p(\varphi_0)$ is an essential $L^1$ module) $\mathcal N_0 \ast \mathcal N_0 \ast \mathcal C_1^p(\varphi_0)$ is indeed dense in $\mathcal C_1^p(\varphi_0)$. But 
    \begin{align*} 
    \mathcal N_0 \ast \mathcal N_0 \ast \mathcal C_1^p(\varphi_0) \subseteq \mathcal N_0 \ast \BUC(\Xi),
    \end{align*}
    showing that $\mathcal N_0 \ast \BUC(\Xi)$ is also dense in $\mathcal C_1^p(\varphi_0)$.
\end{proof}

In particular, $\operatorname{span}\{ (\varphi \otimes \psi) \ast f: ~\varphi, \psi \in W_{\varphi_0}^1, ~f \in \BUC(\Xi)\}$ is a dense subspace of $\mathcal C_1^p(\varphi_0)$. Having this fact at hand, one can prove the following exactly as in \cite[Proposition 3.5]{Fulsche_Luef_Werner2024}.

\begin{proposition} \label{prop:strongly_continuous_extension}
    Let $A \in \mathcal C_1^p(\varphi_0)$. Then, the map $\sigma \Xi \ni x \mapsto \alpha_x(A)$ is continuous in strong operator topology.
\end{proposition}

Note that, similarly to the equality in Proposition \ref{prop:C1N0}, one can prove:

\begin{proposition}\label{char:KCop}
    For each $1 < p < \infty$ it is $\overline{\mathcal N_0 \ast C_0(\Xi)} = \mathcal K(W_{\varphi_0}^p)$. 
\end{proposition}

Indeed, both Propositions \ref{prop:C1N0} and \ref{char:KCop} are particular instances of the \emph{correspondence theorem}, which was described first in \cite{werner84} and was later adapted to $p$-Fock spaces (which essentially are the same as $W_{\varphi_0}^p$ for a certain representation of $\Xi = \mathbb R^{2n}$ and a Gaussian window function) in \cite{Fulsche2020, Fulsche2024} as well as to arbitrary abelian phase spaces (in the Hilbert space setting) in \cite{Fulsche_Luef_Werner2024}. It is no surprise that this result also for $W_{\varphi_0}^p$ on an arbitrary abelian phase space. But since we will make no use of this here, we defer from formulating the general correspondence theory. 

As a consequence of Proposition \ref{char:KCop} we get the compactness characterization of operators on $W_{\varphi_0}^p$ in terms of limit operators. We note that this is a generalization of \cite[Theorem 3.11]{Fulsche_Hagger2025}. The proof is not very different from the one in the setting of\cite{Fulsche_Hagger2025}. But since we ommitted the proof there, we decided to add the short proof here.

\begin{theorem} \label{thm:compactness_characterization}
    Let $A \in \mathcal L(W_{\varphi_0}^p)$. Then, we have $A \in \mathcal K(W_{\varphi_0}^p)$ if and only if $A \in \mathcal C_1^p(\varphi_0)$ and $\alpha_x(A) = 0$ for each $x \in \partial \Xi$. 
\end{theorem}

\begin{proof}
    Let $A \in \mathcal K(W_{\varphi_0}^p)$. Since $\alpha_x(f) = 0$ for each $f \in C_0(\Xi)$, which is readily verified, Proposition \ref{char:KCop} implies $A \in \mathcal C_1^p(\varphi_0)$ and $\alpha_x(A) = 0$ for each $x \in \partial \Xi$. 
    
    On the other hand, let $A \in \mathcal C_1^p(\varphi_0)$ such that $\alpha_x(A) = 0$ for each $x \in \partial \Xi$. Then, $N \ast A \in \BUC(\Xi)$ for each $N \in \mathcal N_0$ and $\alpha_x(N \ast A) = N \ast \alpha_x(A) = 0$ for each $x \in \partial \Xi$. This is well-known to be equivalent to $N \ast A \in C_0(\Xi)$ for each $N \in \mathcal N_0$. Hence, for each $M \in \mathcal N_0$, we have $M \ast N \ast A \in \mathcal K(W_{\varphi_0}^p)$. Since $\{ M \ast N: ~M, N \in \mathcal N_0\}$ is dense in $L^1(\Xi)$, which contains a bounded approximate identity, it is straightforward to show that $A$ can be approximated in operator norm by operators of the form $M \ast N \ast A$. This yields $A \in \mathcal K(W_{\varphi_0}^p)$. 
\end{proof}

\begin{remark}~ \label{rem:compactness_characterization}
    \begin{itemize}
        \item[(a)] Since $V_x \to 0$ weakly as $x \to \partial\Xi$, one can also see directly that $\alpha_x(A) = 0$ for every $A \in \Kc(W_{\varphi_0}^p)$ and $x \in \partial\Xi$. An alternative proof of Theorem \ref{thm:compactness_characterization} without using Proposition \ref{char:KCop} can then be obtained via Theorem \ref{thm:essential_norm} below.
        \item[(b)] Since $\Wc_{\varphi_0} \from \Co_p(U) \to W_{\varphi_0}^p$ is an isometry by definition, all results in this section also translate to limit operators on $\mathrm{Co}_p(U)$. To make this precise, define
        \[\alpha_x(A) := U_xAU_x^{-1}\]
        for $A \in \Lc(\mathrm{Co}_p(U))$ and $x \in \Xi$. Proposition \ref{prop:strongly_continuous_extension} then ensures that if $A \in \Cc_1^p := \Wc_{\varphi_0}^{-1} \Cc_1^p(\varphi_0)\Wc_{\varphi_0}$, then
        \[x \mapsto \alpha_x(A)\]
        extends to a strongly continuous map on $\sigma\Xi$. Moreover, Theorem \ref{thm:compactness_characterization} shows that $A \in \Lc(\mathrm{Co}_p(U))$ is compact if only if $A \in \mathcal C_1^p$ and $\alpha_x(A) = 0$ for each $x \in \partial \Xi$.
    \end{itemize}
\end{remark}

\section{Band-dominated operators}\label{sec:band_dominated}

In this section we provide the definition and some basic properties of band-dominated operators in our metric-free setting. We also show that for every $A \in \mathcal C_1^p(\varphi_0)$ the operator $AP_{\varphi_0} \from L^p(\Xi) \to L^p(\Xi)$ is band-dominated. This will allow us to to apply techniques from the theory of band-dominated operators to our operators $A \in \mathcal C_1^p(\varphi_0)$.

\begin{definition} \label{def:BDO_Xi}
    Let $\Xi$ be an abelian phase space as before and $K \subseteq \Xi$ compact. An operator $A \in \Lc(L^p(\Xi))$ is called a \emph{band operator of band-width at most $K$} if for all $Y \subseteq \Xi$ we have
    \[M_{\1_{\Xi \setminus YK}}AM_{\1_Y} = 0.\]
    The set of all band operators of band-width at most $K$ will be denoted by $\BO_K^p(\Xi)$. The union $\bigcup\limits_{K \subseteq \Xi \text{ compact}} \BO_K(\Xi)$ is denoted by $\BO^p(\Xi)$ and called the \emph{set of band operators}. The closure of $\BO(\Xi)$ with respect to the norm topology is denoted by $\BDO^p(\Xi)$ and its elements are called \emph{band-dominated operators}.
\end{definition}

\begin{remark}~
    \begin{itemize}
        \item[(a)] Definition \ref{def:BDO_Xi} is essentially a metric-free version of \cite[Definition 3.2]{HaggerSeifert}. That is, if we additionally assume that $\Xi$ is second countable, then it is properly metrizable (in the sense that every closed ball is compact) and our definition of band operators becomes equivalent to the usual definition (e.g.~\cite[Definition 3.2]{HaggerSeifert} or \cite[Definition 2.1.5]{Rabinovich_Roch_Silbermann2004}). The compact set $K$ can be chosen as a ball of a certain radius centred at $e$ and the smallest possible radius is then the band width of $A$ in the sense of \cite[Definition 3.2]{HaggerSeifert}. For instance, if $\Xi = \Z$ and $A \in \BO_K(\Z)$, this means that only the diagonals $k \in K$ in the matrix representation of $A$ can be non-zero, where $k = 0$ represents the main diagonal and so on.
        \item[(b)] The sets of band operators of different band-width are partially ordered by inclusion. That is, if $A$ is of band-width at most $K$ and $K \subseteq K'$, then is is also of band-width at most $K'$. It is therefore clear that in the definition of $\BO(\Xi)$ it suffices to take the union over those compact sets $K \subseteq \Xi$ that contain $e$ and satisfy $K = K^{-1}$. This will be useful for a few arguments.
        \item[(c)] Band-dominated operators can of course be defined without the Heisenberg multiplier $m$ and the corresponding representation $(U_x)_{x \in \Xi}$. That is, an abelian phase space is not necessary to do this and some of the results such as Proposition \ref{prop:BDO_algebraic_properties_Xi} below remain valid. However, as we want to combine the operator theoretic techniques with quantum harmonic analysis, we continue with Assumptions \ref{ass_1} and \ref{ass_2} from Section \ref{sec:coorbit}. This also means that $\Xi$ cannot be discrete unless $\Xi$ is finite. Similarly, if $\Xi$ is compact, then it is finite and therefore no interesting results are to be expected in that case.
    \end{itemize}
\end{remark}

$\BDO^p(\Xi)$ has some nice algebraic properties, which we will summarize in the next proposition.

\begin{proposition} \label{prop:BDO_algebraic_properties_Xi}
    $\BO^p(\Xi)$ is an algebra and $\BDO^p(\Xi)$ is a Banach algebra. Moreover, $(\BDO^p(\Xi))^* = \BDO^q(\Xi)$ for $\frac{1}{p} + \frac{1}{q} = 1$, In particular, $\BDO^2(\Xi)$ is a $C^*$-algebra.
\end{proposition}

By $(\BDO^p(\Xi))^* = \BDO^q(\Xi)$ we mean that $A^* \in \BDO^q(\Xi)$ for any $A \in \BDO^p(\Xi)$. We thereby identify $L^q(\Xi)$ with $L^p(\Xi)^*$ via
\[g \mapsto \left(f \mapsto \int_{\Xi} f\overline{g} \, \mathrm{d}\lambda\right).\]

\begin{proof}
    As band operators are partially ordered by inclusion, it is clear that they form a vector space. Now let $K_1,K_2 \subseteq \Xi$ be compact and $A_j \in \BO_{K_j}^p(\Xi)$ for $j = 1,2$. For $Y \subseteq G$ and $f \in L^p(\Xi)$ with $\supp f \subseteq Y$ we get $\supp(A_1f) \in YK_1$ and thus $\supp(A_2A_1f) \in YK_1K_2$. It follows that $A_2A_1 \in \BO_{K_1K_2}^p(\Xi)$. Therefore $\BO^p(\Xi)$ is an algebra. Taking the closure thus yields a Banach algebra. Finally, assume $\frac{1}{p} + \frac{1}{q} = 1$, let $K \subseteq \Xi$ be compact and let $A \in \BO_K^p(\Xi)$. Note that $\big(\Xi \setminus (YK^{-1})\big)K \subseteq \Xi \setminus Y$ for every $Y \subseteq \Xi$. Using the duality between $L^p(\Xi)$ and $L^q(\Xi)$, we thus obtain
    \[0 = \sp{Af}{g} = \sp{f}{A^*g}\]
    for $f \in L^p(\Xi)$, $g \in L^q(\Xi)$ with $\supp f \subseteq \Xi \setminus (YK^{-1})$ and $\supp g \in Y$. In particular, $\supp(A^*g) \subseteq YK^{-1}$. Therefore $A^* \in \BO_{K^{-1}}(\Xi)$. This implies $(\BDO^p(\Xi))^* = \BDO^q(\Xi)$, which also shows that $\BDO^2(G)$ is a $C^*$-algebra.
\end{proof}

$\BDO^p(\Xi)$ is also stable under taking limit operators. In fact, this is even true for each $\BO_K(\Xi)$.

\begin{lemma} \label{lem:BO_limit_ops}
    Let $K \subseteq \Xi$ be compact and $A \in \BO_K(\Xi)$. Then $\alpha_x(A) \in \BO_K(\Xi)$ for all $x \in \sigma\Xi$.
\end{lemma}

\begin{proof}
    For $x \in \Xi$ this follows directly from
    \[\alpha_x(M_{\1_{\Xi \setminus YK}}AM_{\1_Y}) = M_{\1_{\Xi \setminus xYK}}\alpha_x(A)M_{\1_{xY}}\]
    (see Lemma \ref{lem:commutation_properties_V_z}). For $x \in \partial\Xi$ the same follows by taking limits.
\end{proof}

We will need the following crucial result, to merge the methods of band-dominated operators and quantum harmonic analysis.

\begin{theorem}\label{thm:c1_in_bdo}
    The following inclusion holds:
    \[\mathcal C_1^p(\varphi_0) \subseteq P_{\varphi_0} \BDO^p(\Xi) P_{\varphi_0} := \set{P_{\varphi_0}B|_{W^p_{\varphi_0}} : B \in \BDO^p(\Xi)}.\]
\end{theorem}

Note that this means that for every $A \in \mathcal C_1^p(\varphi_0)$ the operator $AP_{\varphi_0} \from L^p(\Xi) \to L^p(\Xi)$ is band-dominated since $P_{\varphi_0}$ is band-dominated as well by Corollary \ref{cor:P_BDO} below.

\begin{remark}
We conjecture that equality holds in Theorem \ref{thm:c1_in_bdo} just like it does for the Fock space; see \cite[Theorem 4.2]{hagger2021}. The main obstacle are explicit estimates on the kernel of $P_{\varphi_0}$ that are not available in the generality of the present paper.
\end{remark}

Proving Theorem \ref{thm:c1_in_bdo} will require several steps.

\begin{lemma}\label{lemma:wiener_in_bdo}
    Let $A  \in \mathcal L(L^p(\Xi))$ be an integral operator with integral kernel $k = k(x,y)$. If there exists $H \in L^1(\Xi)$ such that $|k(x,y)| \leq H(xy^{-1})$ for almost every $x, y \in \Xi$, then $A \in \BDO^p(\Xi)$.
\end{lemma}

\begin{proof}
    Let $\varepsilon > 0$ and $K \subseteq \Xi$ compact such that $\| H \mathbf{1}_{K^c}\|_{L^1} < \varepsilon$. Then, Young's inequality implies
    \begin{align*}
        \| A - A_K\|_{op} \leq \| H \mathbf{1}_{K^c}\|_{L^1} < \varepsilon,
    \end{align*}
    where $A_K$ is the integral operator with kernel $k(x,y) \cdot \mathbf 1_{K}(xy^{-1})$. As $A_K \in \BO_K^p(\Xi)$ follows easily by definition, the result follows.  
\end{proof}

Combining this with Lemma \ref{lemma:projection_wiener_algebra}, we have the following immediate consequence.

\begin{corollary} \label{cor:P_BDO}
    For each $p$ we have $P_{\varphi_0} \in \BDO^p(\Xi)$.
\end{corollary}

For certain integral operators it is easy to see that they are contained in $\BO^p(\Xi)$. Approximating by such operators will be the main step in the proof of Theorem \ref{thm:c1_in_bdo}.

\begin{lemma}\label{lemma:bdo}
    Let $\varphi, \psi \in C_c(\Xi)$ and $f \in L^\infty(\Xi)$. Then the operator $A$ with integral kernel
    \begin{align*}
        k(x,y) = \int_\Xi f(t) \frac{m(t^{-1},y)}{m(t^{-1}, x)} \varphi(xt^{-1}) \overline{\psi(yt^{-1})}~dt
    \end{align*}
    is contained in $\BO^p(\Xi)$.
\end{lemma}

\begin{proof}
    Without loss of generality, we may assume that $\operatorname{supp}(\varphi)$ and $\operatorname{supp}(\psi)$ are both contained in the same compact set $K \subseteq \Xi$. It follows
    \begin{align*}
        |k(x,y)| &\leq  \| f\|_\infty \int_\Xi |\varphi(xt^{-1}) | |\psi(yt^{-1})| ~dt\\
        &= \| f\|_\infty \int_{K} |\varphi(xy^{-1}t)| |\psi(t)|~dt
    \end{align*}
    For $xy^{-1} \not\in KK^{-1}$, the expression on the right-hand side equals zero. For $xy^{-1} \in KK^{-1}$, the expression on the right-hand side is bounded by $\| f\|_\infty \|\varphi\|_{L^2} \| \psi\|_{L^2}$, hence $|k(x,y)| \leq H(xy^{-1})$, where $H(z) = 0$ for $z \not \in KK^{-1}$ and otherwise $\| f\|_\infty \| \varphi\|_{L^2} \| \psi\|_{L^2}$, which is a function in $L^1(\Xi)$. In particular, $A \in \BDO^p(\Xi)$ by Lemma \ref{lemma:wiener_in_bdo}.
\end{proof}

Now we are ready to prove Theorem \ref{thm:c1_in_bdo}.

\begin{proof}[Proof of Theorem \ref{thm:c1_in_bdo}]
    By Proposition \ref{prop:C1N0}, it suffices to prove that
    \[(\varphi \otimes \psi) \ast f \in P_{\varphi_0}\BDO^p(\Xi)P_{\varphi_0}\]
    for $\varphi,\psi \in \mathcal W_{\varphi_0}(\Co_1(U))$, $f \in \BUC(\Xi)$. Since $\mathcal W_{\varphi_0}(\Co_1(U)) = P_{\varphi_0}L^1(\Xi)$, we may assume that $\varphi = P_{\varphi_0} \varphi_1$ and $\psi = P_{\varphi_0} \psi_1$ for some $\varphi_1, \psi_1 \in C_c(\Xi)$. For $a \in \mathcal W_{\varphi_0}(\Co_p(U))$ we have
    \begin{align*}
        \big((P_{\varphi_0} \varphi_1 \otimes P_{\varphi_0} \psi_1) \ast f\big) a &= \int_{\Xi} f(z) V_z  P_{\varphi_0}  \varphi_1 \langle a,  V_z  P_{\varphi_0} \psi_1\rangle~dz\\
        &= P_{\varphi_0} \int_\Xi f(z) (V_z \varphi_1)\otimes (V_z\psi_1)~dz  P_{\varphi_0} a
    \end{align*}
    by Lemma \ref{lem:commutation_properties_V_z}. Now, as an element of $\mathcal L(L^p(\Xi))$ the integral
    \begin{align*}
        \int_\Xi f(z) (V_z \varphi_1) \otimes (V_z \psi_1)~dz
    \end{align*}
    can be expressed for $h \in L^p(\Xi)$ as:
    \begin{align*}
        \left(\int_\Xi f(z) (V_z \varphi_1) \otimes (V_z \psi_1)~dz \; h\right)(x) &= \int_\Xi f(z) V_z \varphi_1 \langle h, V_z \psi_1\rangle ~dz \\
        &= \int_\Xi \int_\Xi h(y) \frac{m(z^{-1}, y)}{m(z^{-1}, x)} f(z) \varphi_1(xz^{-1}) \overline{\psi_1(yz^{-1})}~dz~dy.
    \end{align*}
    By Lemma \ref{lemma:bdo}, this operator is in $\BO^p(\Xi)$, which finishes the proof.
\end{proof}

\section{Operator norm localization for band-dominated operators}\label{sec:norm_localization}

The main purpose of this section is to show that the essential norm of an operator $T \in \Cc_1^p(\varphi_0)$ is equivalent to $\sup\limits_{x \in \partial\Xi} \norm{\alpha_x(T)}$.

For $T \in \Lc(L^p(\Xi))$ and $Y \subseteq \Xi$ we define the localized operator norm as
\[\vertiii{T}_Y := \sup\set{\norm{Tf}_p : \norm{f}_p = 1, \supp f \subseteq x Y \text{ for some } x \in \Xi}.\]
Clearly, we always have $\vertiii{T}_Y \leq \norm{T}$ and $\vertiii{T}_\Xi = \norm{T}$, but the next proposition shows that a compact set $Y \subseteq \Xi$ can be chosen in such a way that $\vertiii{T}_Y$ gets arbitrarily close to $\norm{T}$.

\begin{proposition} \label{prop:ONL}
    Every abelian phase space $\Xi$ has the operator norm localization property. That is, for all $K \subseteq \Xi$ compact and $c \in (0,1)$ there is a compact subset $Y \subseteq \Xi$ such that for all $T \in \BO_K^p(\Xi)$ we have $\vertiii{T}_Y \geq c\norm{T}$. Moreover, we may assume that this $Y$ satisfies $Y = Y^{-1}$ and $e \in Y$.
\end{proposition}

The proof is similar to \cite[Proposition 3.4]{Hagger_Lindner_Seidel2016}. The main ingredient is the fact that lca groups are amenable (see e.g.~\cite[Proposition 0.15]{Paterson1988}). We also note that this result is not surprising in view of \cite[Theorem A.2]{Fulsche_Hagger2025}, which shows that lca groups satisfy property A' introduced in \cite{HaggerSeifert}, which is very much related to the operator norm localization property (see \cite{HaggerSeifert} and \cite{Sako}).

\begin{proof}
    Without loss of generality we may assume that $K = K^{-1}$ and $e \in K$. As $\Xi$ is amenable, \cite[Theorem 4.13]{Paterson1988} shows that for every $\epsilon > 0$ there is a compact subset $Y_0 \subseteq \Xi$ with $\lambda(Y_0) > 0$ such that
    \begin{equation} \label{eq:ONL0}
        \frac{\lambda(Y_0K^2 \Delta Y_0)}{\lambda(Y_0)} < \epsilon.
    \end{equation}
    Here, $\Delta$ denotes the symmetric difference of two sets. For every $f \in L^p(\Xi)$ and $K_0 \subseteq G$ compact we have
    \begin{align*}
       \int_\Xi \norm{\1_{xK_0}f}_p^p \, \mathrm{d}x &= \int_\Xi \int_\Xi \1_{xK_0}(y)\abs{f(y)}^p \, \mathrm{d}y \, \mathrm{d}x\\
       &= \int_\Xi \int_\Xi \1_{yK_0^{-1}}(x)\abs{f(y)}^p \, \mathrm{d}y \, \mathrm{d}x. 
    \end{align*}
    As the integrand is supported on a $\sigma$-finite subset, Tonelli's theorem is applicable and we get
    \begin{equation} \label{eq:ONL1}
        \int_\Xi \norm{\1_{xK_0}f}_p^p \, \mathrm{d}x = \lambda(yK_0^{-1})\norm{f}_p^p = \lambda(K_0)\norm{f}_p^p.
    \end{equation}
    Applying this to $Tf$ and $K_0 := Y_0K$, we get
    \[\lambda(Y_0K)^{\frac{1}{p}}\norm{Tf}_p = \left(\int_\Xi \norm{\1_{xY_0K}Tf}_p^p \, \mathrm{d}x\right)^{\frac{1}{p}} = \norm{\norm{\1_{\cdot Y_0K} Tf}_p}_p.\]
    The triangular equality thus yields
    \begin{equation} \label{eq:ONL2}
        \lambda(Y_0K)^{\frac{1}{p}}\norm{Tf}_p \leq \left(\int_\Xi \norm{[M_{\1_{xY_0K}},T]f}_p^p \, \mathrm{d}x\right)^{\frac{1}{p}} + \left(\int_\Xi \norm{T(\1_{xY_0K}f)}_p^p \, \mathrm{d}x\right)^{\frac{1}{p}}.
    \end{equation}
    For $T \in \BO_K^p(\Xi)$ and $Y_1 \subseteq \Xi$ we now have
    \begin{align*}
        M_{\1_{Y_1K}}T - TM_{\1_{Y_1K}} &= M_{\1_{Y_1K}}TM_{\1_{Y_1K^2}} - TM_{\1_{Y_1K}}\\
        &= M_{\1_{Y_1K}}TM_{\1_{Y_1K^2}} - M_{\1_{Y_1K}}TM_{\1_{Y_1}} + M_{\1_{Y_1K}}TM_{\1_{Y_1}} - TM_{\1_{Y_1K}}\\
        &= M_{\1_{Y_1K}}T(M_{\1_{Y_1K^2}} - M_{\1_{Y_1}}) - TM_{\1_{Y_1K}}(M_{\1_{Y_1K^2}} - M_{\1_{Y_1}})\\
        &= (M_{\1_{Y_1K}}T - TM_{\1_{Y_1K}})M_{\1_{Y_1K^2 \setminus Y_1}}.
    \end{align*}
    Applying this to $Y_1 := xY_0$, we get the estimate
    \[\norm{[M_{\1_{xY_0K}},T]f}_p = \norm{[M_{\1_{xY_0K}},T]M_{\1_{x(Y_0K^2 \setminus Y_0)}}f}_p \leq 2\norm{T}\norm{\1_{x(Y_0K^2 \setminus Y_0)}f}_p.\]
    Therefore, applying \eqref{eq:ONL1} again, the first term in \eqref{eq:ONL2} can be estimated as
    \[\left(\int_\Xi \norm{[M_{\1_{xY_0K}},T]f}_p^p \, \mathrm{d}x\right)^{\frac{1}{p}} \leq 2\norm{T}\lambda(Y_0K^2 \setminus Y_0)^{\frac{1}{p}}\norm{f}_p.\]
    For $\delta \in (0,1)$ choose $f \in L^p(\Xi)$ such that $\norm{f}_p = 1$ and $(1 - \delta)\norm{T} \leq \norm{Tf}_p$. It follows
    \[(1 - \delta)\norm{T}\lambda(Y_0K)^{\frac{1}{p}} \leq 2\norm{T}\lambda(Y_0K^2 \setminus Y_0)^{\frac{1}{p}} + \left(\int_\Xi \norm{T(\1_{xY_0K}f)}_p^p \, \mathrm{d}x\right)^{\frac{1}{p}}\]
    or equivalently
    \[\underbrace{\left(1 - \delta - 2\frac{\lambda(Y_0K^2 \setminus Y_0)^{\frac{1}{p}}}{\lambda(Y_0K)^{\frac{1}{p}}}\right)^p}_{=:c^p}\lambda(Y_0K)\norm{T}^p \leq \int_\Xi \norm{T(\1_{xY_0K}f)}_p^p \, \mathrm{d}x.\]
    Now we apply \eqref{eq:ONL1} again and since $\norm{f}_p = 1$, we obtain
    \[c^p\norm{T}^p\int_\Xi \norm{\1_{xY_0K}f}_p^p \, \mathrm{d}x \leq \int_\Xi \norm{T(\1_{xY_0K}f)}_p^p \, \mathrm{d}x.\]
    In particular, there must be some $x \in \Xi$ with $\1_{xY_0K}f \neq 0$ such that
    \begin{equation} \label{eq:ONL3}
        \norm{T(\1_{xY_0K}f)}_p^p \geq c^p\norm{T}^p\norm{\1_{xY_0K}f}_p^p.
    \end{equation}
    By \eqref{eq:ONL0} we further know that
    \[\frac{\lambda(Y_0K^2 \setminus Y_0)}{\lambda(Y_0K)} \leq \frac{\lambda(Y_0K^2 \Delta Y_0)}{\lambda(Y_0)} < \epsilon.\]
    Choose $Y := Y_0K$. Then $\supp(\1_{xY_0K}f) \subseteq xY$ and since $\delta$ and $\epsilon$ can be chosen arbitrarily small, \eqref{eq:ONL3} implies the proposition.

    That we may choose $Y$ such that $Y = Y^{-1}$ and $e \in Y$ is clear by definition of $\vertiii{\cdot}_Y$.
\end{proof}

In what follows, we will need that we can approximate operator norms somewhat uniformly by the localized versions. This will then be applied to the family of limit operators, which satisfy the assumption in this corollary by Lemma \ref{lem:BO_limit_ops}.

\begin{corollary} \label{cor:ONL}
    Let $\epsilon > 0$ and assume that $(T_j)_{j \in \Jc} \in \BDO^p(\Xi)$ is a bounded family of band-dominated operators. Further assume that there exists a compact subset $K \subseteq \Xi$ and operators $B_j \in \BO_K^p(\Xi)$ such that $\sup\limits_{j \in \Jc} \norm{T_j-B_j} \leq \frac{\epsilon}{3}$. Then there is a compact subset $Y \subseteq \Xi$ such that for every $j \in \Jc$ we have $\vertiii{T_j}_Y \geq \norm{T_j} - \epsilon$. As in Proposition \ref{prop:ONL}, $Y$ can be chosen such that $Y = Y^{-1}$ and $e \in Y$.
\end{corollary}

\begin{proof}
    Set $C := \sup\limits_{j \in \Jc} \norm{B_j}$. By Proposition \ref{prop:ONL}, we can find a compact subset $Y \subseteq \Xi$ such that for every $j \in \Jc$ the inequality $\vertiii{B_j}_Y \geq (1 - \frac{\epsilon}{3C})\norm{B_j}$ is satisfied. It follows
    \[\vertiii{T_j}_Y \geq \vertiii{B_j}_Y - \tfrac{\epsilon}{3} \geq (1 - \tfrac{\epsilon}{C})\norm{B_j} - \tfrac{\epsilon}{3} \geq \norm{B_j} - \tfrac{2\epsilon}{3} \geq \norm{T_j} - \epsilon\]
    for every $j \in \Jc$.
\end{proof}

The following is our main result of this section. For $p = 2$ this follows by a simple $C^*$-argument; see \cite[Lemma 3.9]{Fulsche_Luef_Werner2024}. For $p \neq 2$ the conclusion is not so easy, but essentially follows the lines of \cite[Theorem 4.23]{HaggerSeifert}. Note that this also provides an alternative proof of Theorem \ref{thm:compactness_characterization} (cf.~Remark \ref{rem:compactness_characterization}(a)).

\begin{theorem} \label{thm:essential_norm}
    Let $\Xi$ be an abelian phase space. For $T \in \Cc_1^p(\varphi_0)$ we have
    \[\sup\limits_{x \in \partial \Xi} \norm{\alpha_x(T)} \leq \norm{T + \Kc(W^p_{\varphi_0})} \leq \norm{P_{\varphi_0}} \sup\limits_{x \in \partial \Xi} \norm{\alpha_x(T)}.\]
\end{theorem}

\begin{proof}
As $\alpha_x(A) = 0$ for all $A \in \Kc(W^p_{\varphi_0})$ by Theorem \ref{thm:compactness_characterization} (see also Remark \ref{rem:compactness_characterization}) and $x \in \partial \Xi$, we get
\[\norm{T + A} \geq \norm{\alpha_x(T+A)} = \norm{\alpha_x(T)+\alpha_x(A)} = \norm{\alpha_x(T)}\]
by Lemma \ref{lem:limit_operator_properties}. This implies $\norm{T + \Kc(W^p_{\varphi_0})} \geq \sup\limits_{x \in \partial \Xi} \norm{\alpha_x(T)}$.

Let $\epsilon > 0$ and assume that
\[\norm{T + \Kc(W^p_{\varphi_0})} > \norm{P_{\varphi_0}} \sup\limits_{x \in \partial \Xi} \norm{\alpha_x(T)} + \epsilon.\]
Now note that
\begin{align*}
    \norm{T + \Kc(W^p_{\varphi_0})} &= \inf\limits_{A \in \Kc(W^p_{\varphi_0})} \sup\limits_{\substack{f \in W^p_{\varphi_0} \\ \norm{f}_p \leq 1}} \norm{(T+A)f}\\
    &= \inf\limits_{A \in \Kc(L^p(\Xi),W^p_{\varphi_0})} \sup\limits_{\substack{f \in W^p_{\varphi_0} \\ \norm{f}_p \leq 1}} \norm{(TP_{\varphi_0}+A)f}\\
    &\leq \inf\limits_{A \in \Kc(L^p(\Xi),W^p_{\varphi_0})} \sup\limits_{\substack{f \in L^p(\Xi) \\ \norm{f}_p \leq 1}} \norm{(TP_{\varphi_0}+A)f}\\
    &= \inf\limits_{A \in \Kc(L^p(\Xi),W^p_{\varphi_0})} \norm{TP_{\varphi_0}+A},
\end{align*}
hence $\inf\limits_{A \in \Kc(L^p(\Xi),W^p_{\varphi_0})} \norm{TP_{\varphi_0}+A} > \norm{P_{\varphi_0}} \sup\limits_{x \in \partial \Xi} \norm{\alpha_x(T)} + \epsilon$. As $TP_{\varphi_0}M_{\1_L} \in \Kc(L^p(\Xi),W^p_{\varphi_0})$ for each compact subset $L \subseteq \Xi$ by Corollary \ref{cor:multiplication_compact}, we arrive at
\[\norm{TP_{\varphi_0}M_{\1_{\Xi \setminus L}}} = \norm{TP_{\varphi_0} - TP_{\varphi_0}M_{\1_L}} > \norm{P_{\varphi_0}} \sup\limits_{x \in \partial \Xi} \norm{\alpha_x(T)} + \epsilon.\]
By Theorem \ref{thm:c1_in_bdo}, there is a band operator $B \in \BO_K^p(\Xi)$ such that $\norm{TP_{\varphi_0} - B} \leq \frac{\epsilon}{6}$. It follows that $\norm{TP_{\varphi_0}M_{\1_{\Xi \setminus L}} - BM_{\1_{\Xi \setminus L}}} \leq \frac{\epsilon}{6}$ for all $L \subseteq \Xi$. Moreover, $BM_{\1_{\Xi \setminus L}} \in \BO_K^p(\Xi)$ for all $L \subseteq \Xi$ by Proposition \ref{prop:BDO_algebraic_properties_Xi}. Corollary \ref{cor:ONL} now implies that there is a compact subset $Y \subseteq \Xi$ with $e \in Y$ such that for all $L \subseteq \Xi$ the inequality
\[\vertiii{TP_{\varphi_0}M_{\1_{\Xi \setminus L}}}_Y > \norm{P_{\varphi_0}} \sup\limits_{x \in \partial \Xi} \norm{\alpha_x(T)}  + \frac{\epsilon}{2}\]
is satisfied. In particular, for each compact $L \subseteq \Xi$ we find a $g_L \in \Xi$ such that
\begin{equation} \label{eq:essential_norm1}
    \norm{TP_{\varphi_0}M_{\1_{g_L^{-1}Y}}} \geq \norm{TP_{\varphi_0}M_{\1_{\Xi \setminus L}}M_{\1_{g_L^{-1}Y}}} > \norm{P_{\varphi_0}} \sup\limits_{x \in \partial \Xi} \norm{\alpha_x(T)}  + \frac{\epsilon}{2}.
\end{equation}
Consider the net $(g_L)$ indexed over all compact sets $L \subseteq \Xi$ (partially ordered by inclusion). As $\sigma\Xi$ is compact, $(g_L)$ has a convergent subnet converging to some $x_0 \in \sigma\Xi$. But $x_0$ cannot be contained in $\Xi$ because \eqref{eq:essential_norm1} requires
\[(\Xi \setminus L) \cap g_L^{-1}Y \neq \emptyset,\]
which implies $g_L(\Xi \setminus L) \cap Y \neq \emptyset$ for every compact $L \subseteq \Xi$. It follows $x_0 \in \partial\Xi$. We now take the limit $g_L \to x_0$ in \eqref{eq:essential_norm1}. We know from Proposition \ref{prop:strongly_continuous_extension} that $\alpha_{g_L}(T) \to \alpha_x(T)$ in the strong operator topology and from Corollary \ref{cor:multiplication_compact} that $P_{\varphi_0}M_{\1_Y}$ is compact. This implies that $\alpha_{g_L}(T)P_{\varphi_0}M_{\1_Y} \to \alpha_x(T)P_{\varphi_0}M_{\1_Y}$ in norm by a version of \cite[Theorem 1.1.3]{Rabinovich_Roch_Silbermann2004} for bounded nets. Combining this with $V_{g_L}^{-1}M_{\1_H}V_{g_L} = M_{\1_{g_L^{-1}H}}$ (Lemma \ref{lem:commutation_properties_V_z}) finally yields
\begin{align*}
      \norm{\alpha_{x_0}(T)P_{\varphi_0}M_{\1_Y}} &= \lim\limits_{g_L \to x_0} \norm{\alpha_{g_L}(T)P_{\varphi_0}M_{\1_Y}}\\
      &= \lim\limits_{g_L \to x_0} \norm{TP_{\varphi_0}M_{\1_{g_L^{-1}Y}}}\\
      &\geq \norm{P_{\varphi_0}} \sup\limits_{x_0 \in \partial \Xi} \norm{\alpha_x(T)}  + \frac{\epsilon}{2},  
\end{align*}
which is obviously a contradiction.
\end{proof}

In case $\norm{P_{\varphi_0}} = 1$ we can even show that the supremum is attained.

\begin{theorem} \label{thm:norm_max}
    Let $\Xi$ be an abelian phase space, $T \in \Cc_1^p(\varphi_0)$ and assume that $\norm{P_{\varphi_0}} = 1$ (e.g.,~$p = 2$). Then
    \[\norm{T + \Kc(W^p_{\varphi_0})} = \max\limits_{x \in \partial \Xi} \norm{\alpha_x(T)}.\]
\end{theorem}

\begin{proof}
    In view of Theorem \ref{thm:essential_norm} it is clear that we only need to show that the supremum is a maximum. Let $(x_n)_{n \in \N}$ be a sequence in $\partial\Xi$ such that
    \begin{equation} \label{eq:norm_max0}
        \lim\limits_{n \to \infty} \norm{\alpha_{x_n}(T)} = \sup\limits_{x \in \partial \Xi} \norm{\alpha_x(T)}.
    \end{equation}
     Now note that $\set{\alpha_x(T)P_{\varphi_0}M_{\1_L} : x \in \sigma\Xi, L \subseteq \Xi \text{ Borel}}$ is a suitable family of band-dominated operators in the sense of Corollary \ref{cor:ONL}. Indeed, for every $\epsilon > 0$ we may choose an operator $B \in \BO_K^p(\Xi)$ such that $\norm{TP_{\varphi_0} - B} \leq \frac{\epsilon}{3}$. This implies
    \[\norm{\alpha_x(T)P_{\varphi_0}M_{\1_L} - \alpha_x(B)M_{\1_L}} \leq \norm{\alpha_x(T)P_{\varphi_0} - \alpha_x(B)} \leq \norm{TP_{\varphi_0} - B} \leq \frac{\epsilon}{3}\]
    for all $x \in \sigma\Xi$, $L \subseteq \Xi$. Moreover, $\alpha_x(B)M_{\1_L} \in \BO_K^p(\Xi)$ for all $x \in \sigma\Xi$, $L \subseteq \Xi$ by Proposition \ref{prop:BDO_algebraic_properties_Xi} and Lemma \ref{lem:BO_limit_ops}. Corollary \ref{cor:ONL} thus guarantees the existence of a compact set $Y \subseteq \Xi$ with $Y = Y^{-1}$ and $e \in Y$ and such that for every $x \in \sigma\Xi$ and every $L \subseteq \Xi$ the inequality
    \begin{equation} \label{eq:norm_max1}
        \vertiii{\alpha_x(T)P_{\varphi_0}M_{\1_L}}_Y \geq \norm{\alpha_x(T)P_{\varphi_0}M_{\1_L}} - \epsilon
    \end{equation}
    holds. Applying this to $\epsilon := 2^{-k}$, $k \in \N_0$, we obtain a sequence of compact subsets $Y_k \subseteq \Xi$ with $Y_k = Y_k^{-1}$ and $e \in Y_k$ such that \eqref{eq:norm_max1} holds. Additionally, we may assume, by increasing the size of the sets if necessary, that $\prod\limits_{l = 0}^k Y_l^2 \subseteq Y_{k+1}$ for all $k \in \N_0$. For each $k \in \N_0$ we may choose $\tilde{f}_{k,0} \in L^p(\Xi)$ with $\|\tilde{f}_{k,0}\| \leq 1$, $\supp \tilde{f}_{k,0} \subseteq \xi_{k,0}Y_k$ for some $\xi_{k,0} \in \Xi$ such that
    \begin{equation} \label{eq:norm_max2}
        \norm{\alpha_{x_k}(T)P_{\varphi_0}\tilde{f}_{k,0}} \geq \norm{\alpha_{x_k}(T)P_{\varphi_0}} - 2^{-k}.
    \end{equation}    
    Let $f_{k,0} := V_{\xi_{k,0}^{-1}}\tilde{f}_{k,0}$. Then $\supp f_{k,0} \subseteq Y_k$ and
    \[\norm{\alpha_{\xi_{k,0}^{-1}}\big(\alpha_{x_k}(T)\big)P_{\varphi_0}f_{k,0}} \geq \norm{\alpha_{x_k}(T)P_{\varphi_0}} - 2^{-k}.\]
    Let $z_{k,0} := \xi_{k,0}^{-1}x_k$. For $j = 1,\ldots,k$ we recursively choose the following:
    \begin{itemize}
        \item[(i)] $\tilde{f}_{k,j} \in L^p(\Xi)$ with $\|\tilde{f}_{k,j}\| \leq 1$ and $\supp \tilde{f}_{k,j} \subseteq (\xi_{k,j}Y_{k-j}) \cap Y_{k-j+1}$ for some $\xi_{k,j} \in \Xi$ such that
        \begin{equation} \label{eq:norm_max2.5}
            \norm{\alpha_{z_{k,j-1}}(T)P_{\varphi_0}M_{\1_{Y_{k-j+1}}}\tilde{f}_{k,j}} \geq \norm{\alpha_{z_{k,j-1}}(T)P_{\varphi_0}M_{\1_{Y_{k-j+1}}}} - 2^{-(k-j)},
        \end{equation}
        \item[(ii)] $f_{k,j} := V_{\xi_{k,j}^{-1}}\tilde{f}_{k,j}$, implying $\supp f_{k,j} \subseteq Y_{k-j}$,
        \item[(iii)] $z_{k,j} := \xi_{k,j}^{-1}z_{k,j-1}$.
    \end{itemize}
    This shows
    \begin{align*}
        \norm{\alpha_{z_{k,j}}(T)P_{\varphi_0}M_{\1_{Y_{k-j}}}} &\geq \norm{\alpha_{z_{k,j}}(T)P_{\varphi_0}M_{\1_{Y_{k-j}}}f_{k,j}} = \norm{\alpha_{z_{k,j-1}}(T)P_{\varphi_0}M_{\1_{\xi_{k,j}Y_{k-j}}}\tilde{f}_{k,j}}\\
        &= \norm{\alpha_{z_{k,j-1}}(T)P_{\varphi_0}M_{\1_{Y_{k-j+1}}}\tilde{f}_{k,j}} \geq \norm{\alpha_{z_{k,j-1}}(T)P_{\varphi_0}M_{\1_{Y_{k-j+1}}}} - 2^{-(k-j)}.
    \end{align*}
    In particular,
    \[\norm{\alpha_{z_{k,k-j}}(T)P_{\varphi_0}M_{\1_{Y_j}}} \geq \norm{\alpha_{z_{k,k-j-1}}(T)P_{\varphi_0}M_{\1_{Y_{j+1}}}} - 2^{-j}.\]
    Iterating this estimate and applying \eqref{eq:norm_max2}, we get
    \begin{align*}
        \norm{\alpha_{z_{k,k-j}}(T)P_{\varphi_0}M_{\1_{Y_j}}} &\geq \norm{\alpha_{z_{k,0}}(T)P_{\varphi_0}M_{\1_{Y_k}}} - 2^{-j} - \ldots - 2^{-k-1}\\
        &\geq \norm{\alpha_{z_{k,0}}(T)P_{\varphi_0}} - 2^{-j} - \ldots - 2^{-k-1} - 2^{-k}\\
        &\geq \norm{\alpha_{x_k}(T)P_{\varphi_0}} - 2^{-j+1}
    \end{align*}
    for $j < k$. Now consider $w_k := z_{k,k}$ for $k \in \N$. Note that
    \[z_{k,k-j} = \left(\prod\limits_{l = k-j+1}^k \xi_{k,l}\right)z_{k,k} = \left(\prod\limits_{l = k-j+1}^k \xi_{k,l}\right)w_k.\]
    By \eqref{eq:norm_max2.5}, we necessarily have $\xi_{k,j}^{-1} \in Y_{k-j}Y_{k-j+1}$, which implies that
    \[\left(\prod\limits_{l = k-j+1}^k \xi_{k,l}^{-1}\right)Y_j \subseteq \left(\prod\limits_{l = k-j+1}^k Y_{k-l}Y_{k-l+1}\right)Y_j = \left(\prod\limits_{l = 0}^{j-1} Y_lY_{l+1}\right)Y_j \subseteq Y_{j+1}.\]
    It follows
    \begin{equation} \label{eq:norm_max3}
        \norm{\alpha_{w_k}(T)P_{\varphi_0}M_{\1_{Y_{j+1}}}} \geq \norm{\alpha_{z_{k,k-j}}(T)P_{\varphi_0}M_{\1_{Y_j}}} \geq \norm{\alpha_{x_k}(T)P_{\varphi_0}} - 2^{-j+1}.
    \end{equation}
    As $\partial\Xi$ is compact, $(w_k)_{k \in \N}$ has a subnet $(w_{k_{\gamma}})$ that converges to some $w \in \partial\Xi$. Combining \eqref{eq:norm_max3} with \eqref{eq:norm_max0} and using $\norm{P_{\varphi_0}} = 1$ as well as the fact that $P_{\varphi_0}M_{\1_{Y_{j+1}}}$ is compact, we get
    \begin{align*}
        \norm{\alpha_w(T)} &\geq \norm{\alpha_w(T)P_{\varphi_0}M_{\1_{Y_{j+1}}}} = \lim\limits_{\gamma} \norm{\alpha_{w_{k_{\gamma}}}(T)P_{\varphi_0}M_{\1_{Y_{j+1}}}} \geq \lim\limits_{\gamma} \norm{\alpha_{x_{k_{\gamma}}}(T)P_{\varphi_0}} - 2^{-j+1}\\
        &= \sup\limits_{x \in \partial \Xi} \norm{\alpha_x(T)} - 2^{-j+1}.
    \end{align*}
    As this is true for every $j \in \N$, it follows that
    \[\norm{\alpha_w(T)} = \sup\limits_{x \in \partial \Xi} \norm{\alpha_x(T)}.\qedhere\]
\end{proof}

\section{Fredholm theory}\label{sec:Fredholm}

In this section we combine quantum harmonic analysis tools from \cite{Fulsche_Luef_Werner2024} with $\BDO$-techniques from \cite{HaggerSeifert} to obtain a Fredholm criterion for operators on $\Co_p(U)$ or $W_{\varphi_0}^p$, respectively. For the quantum harmonic analysis part it seems convenient to work with
\[\mathcal C_1^p = \{ A \in \mathcal L(\Co_p(U)): ~\| \alpha_x(A)-A\|_{op} \to 0 \text{ as } x \to 0\}\]
(cf.~Remark \ref{rem:compactness_characterization}(b)). Note that $\mathcal C_1^p \cong \mathcal C_1^p(\varphi_0)$ for each $0 \neq \varphi_1 \in \mathcal H_1$ as Banach algebras, with equivalent norms. We recall from Theorem \ref{thm:essential_norm} that
\[\| A + \mathcal K(\Co_p(U))\| \cong \sup_{x \in \partial \Xi} \| \alpha_x(A)\|.\]

As demonstrated in \cite{Fulsche_Luef_Werner2024} (the arguments presented there did not depend on the Hilbert space structure, which was assumed throughout that paper), for each $A \in \mathcal C_1^p$ one can extend the map
\begin{align*}
    \Xi \ni x \mapsto \alpha_x(A) \in \mathcal C_1^p,
\end{align*}
which is continuous with respect to the operator norm, to a map
\begin{align*}
    \sigma \Xi \ni x \mapsto \alpha_x(A) \in \mathcal C_1^p,
\end{align*}
which is continuous in the strong operator topology. We now make the following definition:
\begin{definition}
    A \emph{compatible family of limit operators}  is a map $\omega: \partial \Xi \to \mathcal C_1^p$ satisfying the following properties:
    \begin{enumerate}[(1)]
        \item $\omega$ is continuous in weak$^\ast$ topology (with respect to the duality $\mathcal L(\Co_p) \cong \mathcal N(\Co_p)$, the nuclear operators).
        \item For every $x \in \partial \Xi$ and $z \in \Xi$ it is $\alpha_z(\omega(x)) = \omega(\varsigma_{z^{-1}}(x))$. Here, $\varsigma_{z^{-1}}(x)$ is the multiplicative linear functional on $\operatorname{BUC}(\Xi)$ defined by $\varsigma_{z^{-1}}(x)(f) = x(\alpha_{z^{-1}}(f))$.
        \item The map $\omega$ is bounded, that is, $\sup\limits_{x \in \partial \Xi}\| \omega(x)\|_{op} < \infty$.
        \item The family $\{ \omega(x): ~x \in \partial \Xi\}$ is uniformly equicontinuous. Here, uniform equicontinuity of a family of operators $\mathcal F$ means that for each $\varepsilon > 0$ there exists a compact neighborhood $K \subseteq \Xi$ of the identity $e \in \Xi$ such that for each $z \in K$ and each $A \in \mathcal F$ we have $\| A - \alpha_z(A)\| < \varepsilon$.
    \end{enumerate}
\end{definition}
Of course, the definition is directly taken from \cite[Definition 3.6]{Fulsche_Luef_Werner2024}, with the sole difference that we also work outside of the Hilbert space setting. The set of all compatible families of limit operators will be denoted $\mathfrak{lim}\, \mathcal C_1^p$. There is one main result about compatible families of limit operators, which is the following: 

\begin{theorem}\label{thm:comp_families}~
    \begin{enumerate}[(1)]
        \item Let $B \in \mathcal C_1^p$. Then, $\omega(x) := \alpha_x(B)$, $x\in \partial \Xi$, is a compatible family of limit operators.
        \item If $\omega$ is a compatible family of limit operators, then there exists $B \in \mathcal C_1^p$ such that $\omega(x) = \alpha_x(B)$. $B$ is unique modulo $\mathcal K(\Co_p(U))$. 
        \item Upon endowing $\mathfrak{lim}\, \mathcal C_1^p$ with pointwise addition, multiplication and the norm
        \[\| \omega\| := \sup\limits_{x \in \partial \Xi} \| \omega(x)\|_{op},\]
        it turns into a unital Banach algebra.
        \item The map $\mathcal C_1^p \ni B \mapsto [\omega(x) = \alpha_x(B)] \in \mathfrak{lim}\, \mathcal C_1^p$ is a surjective homomorphism of Banach algebras. Its kernel is $\mathcal K(\Co_p(U))$ and the quotient map $\mathcal C_1^p/\mathcal K(\Co_p(U)) \to \mathfrak{lim}\, \mathcal C_1^p$ is, up to an equivalence of norms, an isomorphism of Banach algebras.
    \end{enumerate}
\end{theorem}
The proof of the theorem is essentially the same as in \cite[Theorem 3.7]{Fulsche_Luef_Werner2024}. Indeed, the only argument used there depending on the Hilbert space structure was the fact that $\| A + \mathcal K\| = \sup\limits_{x \in \partial \Xi} \| \alpha_x(A)\|$, which was in the Hilbert space case deduced from the general theory of $C^\ast$-algebras. Here, we have Theorem \ref{thm:essential_norm} as a substitute, which is sufficient to now follow the same proof.

We have the following consequence of Theorem \ref{thm:comp_families}:
\begin{corollary} \label{cor:comp_families}
    Let $A \in \mathcal C_1^p$. Then, $A$ is Fredholm if and only if $\alpha_x(A)$ is invertible for each $x \in \partial \Xi$ and $\sup\limits_{x \in \partial \Xi} \| \alpha_x(A)^{-1}\| < \infty$.
\end{corollary}

\begin{proof}
    The proof can be concluded as in the Fock space case, cf.\ \cite[Theorem 6.15]{Fulsche2024}. For completeness, we repeat the argument. If $A$ is Fredholm, the statement follows from the fact that the Fredholm regularizer of an operator in $\mathcal C_1^p$ is again contained in $\mathcal C_1^p$: This is a consequence of the fact that $\mathcal K(\Co_p(U)) \subseteq \mathcal C_1^p$. Since the shifts $\alpha_x$ leave $\mathcal K(\Co_p(U))$ invariant, we can consider the shifts also on the Calkin algebra $\mathcal L(\Co_p(U))/\mathcal K(\Co_p(U))$, and $\mathcal C_1^p/\mathcal K(\Co_p(U))$ can be considered as the subalgebra of the Calkin algebra on which the shifts act uniformly continuous, i.e.:
    \begin{align*}
        \mathcal C_1^p/\mathcal K(\Co_p(U)) = \{ &A + \mathcal K(\Co_p(U)) \in \mathcal L(\Co_p(U))/\mathcal K(\Co_p(U)): \\
        &z \mapsto \alpha_z(A + \mathcal K(\Co_p(U))) \text{ is cont.\ w.r.t.} \| \cdot\|_{\mathcal L/\mathcal K}\}.
    \end{align*}
    Therefore, a standard Neumann series argument shows that Fredholm regularizers of $A \in \mathcal C_1^p$, which can be identified with inverses in the Calkin algebra, are again contained in $\mathcal C_1^p$. 
    
    On the other hand, if $A \in \mathcal C_1^p$ is such that $\sup\limits_{x \in \partial \Xi} \| \alpha_x(A)^{-1}\| < \infty$, then $\omega(x) := \alpha_x(A)^{-1}$ defines a compatible family of limit operators: Weak$^\ast$ continuity of the function follows from the second resolvent identity. (2) follows from the fact that $\alpha_z$ commutes with taking inverses. (3) is clear and the uniform equicontinuity follows from a standard Neumann series argument. Hence, by Theorem \ref{thm:comp_families}(4), there exists some $B \in \mathcal C_1^p$ such that $\alpha_x(B) = \alpha_x(A)^{-1}$ for each $x \in \partial \Xi$. It follows $\alpha_x(AB - I) = 0$ for each $x \in \partial \Xi$, which shows $AB - I \in \mathcal K(\Co_p(U))$. Similarly, $BA - I \in \mathcal K(\Co_p(U))$.
\end{proof}

It now remains to remove the condition on the uniform boundedness of the inverses of the limit operators, that is, to extend the above corollary to the following main result about Fredholmness:

\begin{theorem} \label{thm:Fredholm_main}
    Let $1 < p < \infty$ and $A \in \mathcal C_1^p$. Then $A$ is Fredholm if and only if $\alpha_x(A)$ is invertible for every $x \in \partial \Xi$. In particular, the essential spectrum of $A$ is given by
\[\spec_{\ess}(A) = \bigcup\limits_{x \in \partial\Xi} \spec(\alpha_x(A)).\]
\end{theorem}

To prove this result, we switch our point of view to $\Cc_1^p(\varphi_0)$, which acts on $W^p_{\varphi_0} \subseteq L^p(\Xi)$. For $T \in \Lc(L^p(\Xi))$ and $Y \subseteq \Xi$ we define the lower norm
\[\nu(T) := \inf\set{\norm{Tf}_p : \norm{f}_p = 1}\]
and the localized lower norm as
\[\nu_Y(T) := \inf\set{\norm{Tf}_p : \norm{f}_p = 1, \supp f \subseteq \xi Y \text{ for some } \xi \in \Xi}.\]
We will also need a restricted version, which for convenience we will denote by
\[\nu_Y(T|_L) := \inf\set{\norm{Tf}_p : \norm{f}_p = 1, \supp f \subseteq \xi Y \cap L \text{ for some } \xi \in \Xi}\]
for a Borel subset $L \subseteq \Xi$. Clearly, we always have $\nu_Y(T) \geq \nu(T)$ and $\nu_Y(T|_L) \geq \nu_Y(T)$. Similarly as for the operator norm, a compact set $Y \subseteq \Xi$ can be chosen in such a way that $\nu_Y(T|_L)$ gets arbitrarily close to $\nu(T|_L)$.

\begin{proposition} \label{prop:LNL}
    Every abelian phase space $\Xi$ has the lower norm localization property. That is, for all $K \subseteq \Xi$ compact, $C \geq 0$ and $\epsilon > 0$ there is a compact subset $Y \subseteq \Xi$ such that for all $T \in \BO_K^p(\Xi)$ with $\norm{T} \leq C$ and all Borel subsets $L \subseteq \Xi$ we have $\nu_Y(T|_L) \leq \nu(T|_L) + \epsilon$. Moreover, we may assume that this $Y$ satisfies $Y = Y^{-1}$ and $e \in Y$.
\end{proposition}

Note that compared to Proposition \ref{prop:ONL}, where a similar result is proven for the operator norm, the approximation error is an additive constant instead of a multiplicative one. This happens essentially because unlike for the operator norm, the lower norm can be zero without the operator itself being zero. Nevertheless, the proof is still pretty much the same as for the operator norm.

\begin{proof}
    As in the proof of Proposition \ref{prop:ONL}, for every $\epsilon > 0$ there is a compact subset $Y_0 \subseteq \Xi$ such that
    \begin{equation} \label{eq:LNL0}
        \frac{\lambda(Y_0K^2 \Delta Y_0)}{\lambda(Y_0)} < \epsilon.
    \end{equation}
    Let $L \subseteq \Xi$ be a Borel set and $f \in L^p(\Xi)$ with $\supp f \subseteq L$. Using \eqref{eq:ONL1}, we obtain
    \[\lambda(Y_0K)^{\frac{1}{p}}\norm{Tf}_p = \left(\int_\Xi \norm{\1_{xY_0K}Tf}_p^p \, \mathrm{d}x\right)^{\frac{1}{p}} = \norm{\norm{\1_{\cdot Y_0K} Tf}_p}_p.\]
    But now we apply the reverse triangular equality to get
    \[\lambda(Y_0K)^{\frac{1}{p}}\norm{Tf}_p \geq \left(\int_\Xi \norm{T(\1_{xY_0K}f)}_p^p \, \mathrm{d}x\right)^{\frac{1}{p}} - \left(\int_\Xi \norm{[M_{\1_{xY_0K}},T]f}_p^p \, \mathrm{d}x\right)^{\frac{1}{p}}.\]
    The second term may again be estimated as
    \[\left(\int_\Xi \norm{[M_{\1_{xY_0K}},T]f}_p^p \, \mathrm{d}x\right)^{\frac{1}{p}} \leq 2\norm{T}\lambda(Y_0K^2 \setminus Y_0)^{\frac{1}{p}}\norm{f}_p.\]
    For $\delta > 0$ choose $f \in L^p(\Xi)$ with $\supp f \subseteq L$ such that $\norm{f}_p = 1$ and $\nu(T|_L) + \delta \geq \norm{Tf}_p$. It follows
    \[(\nu(T|_L) + \delta)\lambda(Y_0K)^{\frac{1}{p}}\geq \left(\int_\Xi \norm{T(\1_{xY_0K}f)}_p^p \, \mathrm{d}x\right)^{\frac{1}{p}} - 2\norm{T}\lambda(Y_0K^2 \setminus Y_0)^{\frac{1}{p}}\]
    or equivalently
    \[\Bigg(\nu(T|_L) + \underbrace{\delta + 2\norm{T}\frac{\lambda(Y_0K^2 \setminus Y_0)^{\frac{1}{p}}}{\lambda(Y_0K)^{\frac{1}{p}}}}_{=:\epsilon'}\Bigg)^p\lambda(Y_0K) \geq \int_\Xi \norm{T(\1_{xY_0K}f)}_p^p \, \mathrm{d}x.\]
    Now we apply \eqref{eq:ONL1} again and since $\norm{f}_p = 1$, we obtain
    \[(\nu(T|_L) + \epsilon')^p\int_\Xi \norm{\1_{xY_0K}f}_p^p \, \mathrm{d}x \geq \int_\Xi \norm{T(\1_{xY_0K}f)}_p^p \, \mathrm{d}x.\]
    In particular, there must be some $x \in \Xi$ such that $\1_{xY_0K}f \neq 0$ and
    \begin{equation} \label{eq:LNL3}
        (\nu(T|_L) + \epsilon')^p\norm{\1_{xY_0K}f}_p^p \geq \norm{T(\1_{xY_0K}f)}_p^p.
    \end{equation}
    By \eqref{eq:LNL0} we further know that
    \[\frac{\lambda(Y_0K^2 \setminus Y_0)}{\lambda(Y_0K)} \leq \frac{\lambda(Y_0K^2 \Delta Y_0)}{\lambda(Y_0)} < \epsilon.\]
    Now choose $Y := Y_0K$, which implies $\supp(\1_{xY_0K}f) \subseteq xY \cap L$. As $\delta$ and $\epsilon$ can be chosen arbitrarily small, \eqref{eq:LNL3} implies the proposition.

    That we may choose $Y$ such that $Y = Y^{-1}$ and $e \in Y$ is clear by definition of $\nu_Y$.
\end{proof}

As for the norm in Corollary \ref{cor:ONL}, we will need a certain uniformity in the approximation of the lower norm.

\begin{corollary} \label{cor:LNL}
    Let $\epsilon > 0$ and assume that $(T_j)_{j \in \Jc} \in \BDO^p(\Xi)$ is a bounded family of band-dominated operators. Further assume that there exists a compact subset $K \subseteq \Xi$ and operators $B_j \in \BO_K^p(\Xi)$ such that $\sup\limits_{j \in \Jc} \norm{T_j-B_j} \leq \frac{\epsilon}{3}$. Then there is a compact subset $Y \subseteq \Xi$ such that for every $j \in \Jc$ and every Borel subset $L \subseteq \Xi$ we have $\nu_Y(T_j|_L) \leq \nu(T_j|_L) + \epsilon$. As in Proposition \ref{prop:ONL}, $Y$ can be chosen such that $Y = Y^{-1}$ and $e \in Y$.
\end{corollary}

\begin{proof}
    By Proposition \ref{prop:LNL}, we can find a compact subset $Y \subseteq \Xi$ such that for every $j \in \Jc$ and every Borel subset $L \subseteq \Xi$ the inequality $\nu_Y(B_j|_L) \leq \nu(B_j|_L) + \frac{\epsilon}{3}$ is satisfied. Moreover, it is easy to see that
    \[\abs{\nu_Y(T_j|_L) - \nu_Y(B_j|_L)} \leq \norm{T_j - B_j} \leq \tfrac{\epsilon}{3}\]
    (see e.g. \cite[Lemma 2.38]{Lindner2006}). It follows
    \[\nu_Y(T_j|_L) \leq \nu_Y(B_j|_L) + \tfrac{\epsilon}{3} \leq \nu(B_j|_L) + \tfrac{2\epsilon}{3} \leq \nu(T_j|_L) + \epsilon\]
    for every $j \in \Jc$.
\end{proof}

For $T \in \Cc_1^p(\varphi_0)$ we define its extension to $L^p(\Xi)$ by $\hat{T} := TP_{\varphi_0} + (I - P_{\varphi_0})$. As $P_{\varphi_0}$ commutes with $V_x$ by Lemma \ref{lem:commutation_properties_V_z}, we have $\alpha_x(\hat{T}) = \alpha_x(T)P_{\varphi_0} + (I - P_{\varphi_0})$ for all $x \in \sigma\Xi$.

\begin{lemma} \label{lem:lower_norm_minimum}
    Let $T \in \Cc_1^p(\varphi_0)$. Then there is an $x \in \partial\Xi$ such that
    \[\nu(\alpha_x(\hat{T})) = \inf\set{\nu(\alpha_z(\hat{T})) : z \in \partial\Xi}.\]
\end{lemma}

The proof is quite similar to the proof of Theorem \ref{thm:norm_max}, but the extension is needed because otherwise the lower norm just vanishes if nothing is added to the complement of $W^p_{\varphi_0}$.

\begin{proof}
    Let $(x_n)_{n \in \N}$ be a sequence in $\partial\Xi$ such that
    \begin{equation} \label{eq:nu_max0}
        \lim\limits_{n \to \infty} \nu(\alpha_{x_n}(\hat{T})) = \inf\limits_{x \in \partial \Xi} \nu(\alpha_x(\hat{T})).
    \end{equation}
    By Lemma \ref{lem:BO_limit_ops} and Corollary \ref{cor:LNL}, there is a compact set $Y \subseteq \Xi$ with $e \in Y$ and $Y = Y^{-1}$ such that
    \begin{equation} \label{eq:nu_max1}
        \nu_Y(\alpha_x(\hat{T})|_L) \leq \nu(\alpha_x(\hat{T})|_L) + \epsilon.
    \end{equation}
    for all $x \in \Xi$ and all Borel sets $L \subseteq \Xi$. Applying this to $\epsilon := 2^{-k}$, $k \in \N_0$, we obtain a sequence of compact subsets $Y_k \subseteq \Xi$ with $Y_k = Y_k^{-1}$ and $e \in Y_k$ such that \eqref{eq:nu_max1} holds. As in the proof of Theorem \ref{thm:norm_max}, we may assume that $\prod\limits_{l = 0}^k Y_l^2 \subseteq Y_{k+1}$ for all $k \in \N_0$. For each $k \in \N_0$ we may choose $\tilde{f}_{k,0} \in L^p(\Xi)$ with $\|\tilde{f}_{k,0}\| \leq 1$, $\supp \tilde{f}_{k,0} \subseteq \xi_{k,0}Y_k$ for some $\xi_{k,0} \in \Xi$ such that
    \begin{equation} \label{eq:nu_max2}
        \norm{\alpha_{x_k}(\hat{T})\tilde{f}_{k,0}} \leq \nu(\alpha_{x_k}(\hat{T})) + 2^{-k}.
    \end{equation}    
    Let $f_{k,0} := V_{\xi_{k,0}^{-1}}\tilde{f}_{k,0}$. Then $\supp f_{k,0} \subseteq Y_k$ and
    \[\norm{\alpha_{\xi_{k,0}^{-1}}\big(\alpha_{x_k}(\hat{T})\big)f_{k,0}} \leq \nu(\alpha_{x_k}(\hat{T})) + 2^{-k}.\]
    Let $z_{k,0} := \xi_{k,0}^{-1}x_k$. For $j = 1,\ldots,k$ we recursively choose the following:
    \begin{itemize}
        \item[(i)] $\tilde{f}_{k,j} \in L^p(\Xi)$ with $\|\tilde{f}_{k,j}\| \leq 1$ and $\supp \tilde{f}_{k,j} \subseteq (\xi_{k,j}Y_{k-j}) \cap Y_{k-j+1}$ for some $\xi_{k,j} \in \Xi$ such that
        \begin{equation} \label{eq:nu_max2.5}
            \norm{\alpha_{z_{k,j-1}}(\hat{T})\tilde{f}_{k,j}} \leq \nu(\alpha_{z_{k,j-1}}(\hat{T})|_{Y_{k-j+1}}) + 2^{-(k-j)},
        \end{equation}
        \item[(ii)] $f_{k,j} := V_{\xi_{k,j}^{-1}}\tilde{f}_{k,j}$, implying $\supp f_{k,j} \subseteq Y_{k-j}$,
        \item[(iii)] $z_{k,j} := \xi_{k,j}^{-1}z_{k,j-1}$.
    \end{itemize}
    This shows
    \[\nu(\alpha_{z_{k,j}}(\hat{T})|_{Y_{k-j}}) \leq \norm{\alpha_{z_{k,j}}(\hat{T})f_{k,j}} = \norm{\alpha_{z_{k,j-1}}(\hat{T})\tilde{f}_{k,j}} \leq \nu(\alpha_{z_{k,j-1}}(\hat{T})|_{Y_{k-j+1}}) + 2^{-(k-j)}.\]
    In particular,
    \[\nu(\alpha_{z_{k,k-j}}(\hat{T})|_{Y_j}) \leq \nu(\alpha_{z_{k,k-j-1}}(\hat{T})|_{Y_{j+1}}) + 2^{-j}.\]
    Iterating this estimate and applying \eqref{eq:nu_max2}, we get
    \begin{align*}
        \nu(\alpha_{z_{k,k-j}}(\hat{T})|_{Y_j}) &\leq \nu(\alpha_{z_{k,0}}(\hat{T})|_{Y_k}) + 2^{-j} + \ldots + 2^{-k-1}\\
        &\leq \nu(\alpha_{z_{k,0}}(\hat{T})) + 2^{-j} + \ldots + 2^{-k-1} + 2^{-k}\\
        &\leq \nu(\alpha_{x_k}(\hat{T})) + 2^{-j+1}
    \end{align*}
    for $j < k$. Now consider $w_k := z_{k,k}$ for $k \in \N$. As in the proof of Theorem \ref{thm:norm_max}, it follows
    \begin{equation} \label{eq:nu_max3}
        \nu(\alpha_{w_k}(\hat{T})|_{Y_{j+1}}) \leq \nu(\alpha_{z_{k,k-j}}(\hat{T})|_{Y_j}) \leq \nu(\alpha_{x_k}(\hat{T})) + 2^{-j+1}.
    \end{equation}
    As $\partial\Xi$ is compact, $(w_k)_{k \in \N}$ has a subnet $(w_{k_{\gamma}})$ that converges to some $w \in \partial\Xi$. As $P_{\varphi_0}M_{\1_{Y_{j+1}}}$ is compact by Corollary \ref{cor:multiplication_compact}, we get that $\alpha_{w_{k_{\gamma}}}(T)P_{\varphi_0}M_{\1_{Y_{j+1}}} \to \alpha_w(T)P_{\varphi_0}M_{\1_{Y_{j+1}}}$ in operator norm. This clearly implies $\nu(\alpha_{w_{k_{\gamma}}}(\hat{T})|_{Y_{j+1}}) \to \nu(\alpha_w(\hat{T})|_{Y_{j+1}})$ (see e.g. \cite[Lemma 2.38]{Lindner2006}). Combining this observation with \eqref{eq:nu_max0} and \eqref{eq:nu_max3}, we get
    \begin{align*}
        \nu(\alpha_w(\hat{T})) &\leq \nu(\alpha_w(\hat{T})|_{Y_{j+1}}) = \lim\limits_{\gamma} \nu(\alpha_{w_{k_{\gamma}}}(\hat{T})|_{Y_{j+1}}) \leq \lim\limits_{\gamma} \nu(\alpha_{x_{k_{\gamma}}}(\hat{T})) + 2^{-j+1}\\
        &= \sup\limits_{x \in \partial \Xi} \nu(\alpha_x(\hat{T})) + 2^{-j+1}.
    \end{align*}
    As this is true for every $j \in \N$, it follows that
    \[\nu(\alpha_w(T)) = \inf\limits_{x \in \partial \Xi} \nu(\alpha_x(T)).\qedhere\]
\end{proof}

Our main result now follows by combining the above preparations.

\begin{proof}[Proof of Theorem \ref{thm:Fredholm_main}]
    In view of Corollary \ref{cor:comp_families}, it suffices to prove that if $\alpha_x(T)$ is invertible for each $x \in \partial \Xi$, then $\sup\limits_{x \in \partial \Xi} \| \alpha_x(T)^{-1}\| < \infty$. Now note that $\nu(T) = \norm{T^{-1}}^{-1}$ if $T$ is invertible. Moreover, if $T$ is invertible, then so is $\hat{T}$. So assuming that $\alpha_x(T)$ is invertible for each $x \in \partial \Xi$, we get that $\nu(\alpha_x(\hat{T})) > 0$ for every $x \in \partial\Xi$. Lemma \ref{lem:lower_norm_minimum} thus implies that $\inf\limits_{x \in \partial\Xi} \nu(\alpha_x(\hat{T})) > 0$ and hence $\sup\limits_{x \in \partial \Xi} \| \alpha_x(\hat{T})^{-1}\| < \infty$. As
    \[\alpha_x(\hat{T})^{-1} = \alpha_x(T)^{-1}P_0 + (I - P_0),\]
    this also shows $\sup\limits_{x \in \partial \Xi} \| \alpha_x(T)^{-1}\| < \infty$.
\end{proof}

\bibliographystyle{abbrv}
\bibliography{main}

@article{werner84,
	title = {{Quantum Harmonic Analysis on Phase Space}},
	author = {Werner, R.},
	year = {1984},
	journal = {J. Math. Phys.},
	volume = {25},
	issue = {5},
	pages = {1404–1411},
}

@article{hagger2021,
	author = {Hagger, R.},
	title = {{Essential Commutants and Characterizations of the Toeplitz Algebra}},
	journal = {J. Operator Theory},
	volume = {86},
	issue = {1},
	pages = {125-143},
	year = {2021},
}

@ARTICLE{Fulsche2020,
	author = {Fulsche, R.},
	title = {{Correspondence theory on $p$-Fock spaces with applications to Toeplitz algebras}},
	journal = {J. Funct. Anal.},
	volume = {279},
	issue = {7},
	pages = {108661},
	year = {2020},
}

@article{HaggerSeifert,
    title = {Limit operators techniques on general metric measure spaces of bounded geometry},
    author = {Hagger, R. and Seifert, C.},
    journal = {J. Math. Anal. Appl.},
    volume = 489,
    issue = 2,
    year = 2020,
    pages = {124180},
}

@article{Luef_Skrettingland2021,
    title = {{A Wiener tauberian theorem for operators and functions}},
    author = {Luef, F. and Skrettingland, E.},
    journal = {J. Funct. Anal.},
    volume = {280},
    pages = {108883},
    year = {2021},
}

@article{Fulsche_Galke2023,
 author = {Fulsche, R. and Galke, N.},
 title = {Quantum harmonic analysis on locally compact abelian groups},
 fjournal = {The Journal of Fourier Analysis and Applications},
 journal = {J. Fourier Anal. Appl.},
 issn = {1069-5869},
 volume = {31},
 number = {1},
 pages = {58},
 note = {Id/No 13},
 year = {2025},
}

@article {Halvdansson2023,
    author = {Halvdansson, Simon},
    title = {Quantum harmonic analysis on locally compact groups},
    journal = {J. Funct. Anal.},
    volume = {285},
    issue = {8},
    year = {2023},
    pages = {110096},
}

@book{Rabinovich_Roch_Silbermann2004,
	author = {Rabinoch, V.~S. and Roch, S. and Silbermann, B.},
	title = {{Limit Operators and Their Applications in Operator Theory}},
	series = {Oper. Theory Adv. Appl.},
	volume = {150},
	publisher = {Birkh\"auser},
	year = {2004},
}

@book{Paterson1988,
    author = {Paterson, A.~L.~T.},
    title = {Amenability},
    year = {1988},
    series = {Math. Surveys Monogr.},
    volume = {29},
    publisher = {American Mathematical Society},
    address = {Providence, RI},
}

@article{Fulsche_Luef_Werner2024,
    author = {Fulsche, R. and Luef, F. and Werner, R.~F.},
    title = {{Wiener's Tauberian theorem in classical and quantum harmonic analysis}},
    journal = {J. Funct. Anal.},
    year = {2026},
    volume = 290,
    issue = 4,
    pages = {111265},
}

@article{Fulsche2024,
    author = {Fulsche, R.},
    title = {{Toeplitz operators on non-reflexive Fock spaces}},
    journal = {Rev. Math. Iberoam.},
    volume = 40,
    issue = 3,
    pages = {1115–1148},
    year = 2024,
}

@article{Hagger_Lindner_Seidel2016,
    author = {Hagger, R. and Lindner, M. and Seidel, M.},
    title = {Essential pseudospectra and essential norms of band-dominated operators},
    journal = {J. Math. Anal. Appl.},
    volume = {437},
    year = {2016},
    number = {1},
    pages = {255--291},
}

@article{Persson1963,
    author = {Persson, A.},
    title = {{Compact linear mappings between interpolation spaces}},
    journal = {Ark. Mat.},
    year = {1964 - 1965},
    volume = 5,
    issue = {3-4},
    pages = {215-219},
}

@article{Cwikel1992,
    author = {Cwikel, M.},
    title = {Real and complex interpolation and extrapolation of compact operators},
    journal = {Duke Math. J.},
    year = 1992,
    volume = 65,
    issue = 2,
    pages = {333-343},
}

@book{Zaanen1997,
    author = {Zaanen, A.~C.},
    title = {{Introduction to Operator Theory in Riesz Spaces}},
    year = 1997,
    publisher = {Springer},
}

@article{Baggett_Kleppner1973,
    author = {Baggett, L. and Kleppner, A.},
    title = {{Multiplier Representations of Abelian Groups}},
    year = 1973,
    journal = {J. Funct. Anal.},
    volume = 14,
    issue = 3,
    pages = {299-324},
}

@book{Lindner2006,
    author = {Lindner, M.},
    title = {{Infinite Matrices and Their Finite Sections}},
    year = 2006,
    publisher = {Birkh\"auser Verlag},
    address = {Basel, Boston, Berlin},
}

@misc{Fulsche_Hagger2025,
 author = {Fulsche, R. and Hagger, R.},
 title = {Band-dominated and {Fourier}-band-dominated operators on locally compact abelian groups},
 year = {2025},
 howpublished = {Preprint, {arXiv}:2504.17442 [math.{FA}] (2025)},
}

@article{Feichtinger_Groechenig1989a,
 author = {Feichtinger, H.~G. and Gr{\"o}chenig, K.},
 title = {Banach spaces related to integrable group representations and their atomic decompositions. {I}},
 journal = {J. Funct. Anal.},
 volume = {86},
 number = {2},
 pages = {307--340},
 year = {1989},
}

@article{Feichtinger_Groechenig1989b,
 author = {Feichtinger, H.~G. and Gr{\"o}chenig, K.},
 title = {Banach spaces related to integrable group representations and their atomic decompositions. {II}},
 fjournal = {Monatshefte f{\"u}r Mathematik},
 journal = {Monatsh. Math.},
 volume = {108},
 number = {2-3},
 pages = {129--148},
 year = {1989},
}

@article{Christensen1996,
 author = {Christensen, O.},
 title = {Atomic decomposition via projective group representations},
 fjournal = {Rocky Mountain Journal of Mathematics},
 journal = {Rocky Mt. J. Math.},
 issn = {0035-7596},
 volume = {26},
 number = {4},
 pages = {1289--1312},
 year = {1996},
}

@article{Hagger2017,
 author = {Hagger, R.},
 title = {The essential spectrum of {Toeplitz} operators on the unit ball},
 fjournal = {Integral Equations and Operator Theory},
 journal = {Integral Equations Oper. Theory},
 issn = {0378-620X},
 volume = {89},
 number = {4},
 pages = {519--556},
 year = {2017},
 language = {English},
}

@article{Hagger2019,
 author = {Hagger, R.},
 title = {Limit operators, compactness and essential spectra on bounded symmetric domains},
 fjournal = {Journal of Mathematical Analysis and Applications},
 journal = {J. Math. Anal. Appl.},
 issn = {0022-247X},
 volume = {470},
 number = {1},
 pages = {470--499},
 year = {2019},
 language = {English},
}

@article{Fulsche_Hagger2019,
 author = {Fulsche, R. and Hagger, R.},
 title = {Fredholmness of {Toeplitz} operators on the {Fock} space},
 fjournal = {Complex Analysis and Operator Theory},
 journal = {Complex Anal. Oper. Theory},
 issn = {1661-8254},
 volume = {13},
 number = {2},
 pages = {375--403},
 year = {2019},
 language = {English},
}

@article{Prasad_Shapiro_Vemuri2010,
 author = {Prasad, A. and Shapiro, I. and Vemuri, M.~K.},
 title = {Locally compact abelian groups with symplectic self-duality},
 fjournal = {Advances in Mathematics},
 journal = {Adv. Math.},
 issn = {0001-8708},
 volume = {225},
 number = {5},
 pages = {2429--2454},
 year = {2010},
}

@ARTICLE{Lindner_Seidel,
	author = {Lindner, M. and Seidel, M.},
	title = "{An affirmative answer to a core issue on limit operators}",
	year = {2014},
	journal = {J. Funct. Anal.},
	volume = {267},
	issue = {3},
	pages = {901-917},
}

@ARTICLE{Spakula_Willett,
	author = {\v{S}pakula, J. and Willett, R.},
	title = "{A metric approach to limit operators}",
	year = {2017},
	journal = {Trans. Am. Math. Soc.},
	volume = {369},
	pages = {263-308},
}

@article{Fulsche_Hagger2024,
 author = {Fulsche, R. and Hagger, R.},
 title = {{Quantum Harmonic Analysis for Polyanalytic Fock Spaces}},
 fjournal = {The Journal of Fourier Analysis and Applications},
 journal = {J. Fourier Anal. Appl.},
 volume = {30},
 number = {6},
 pages = {1-42},
 year = {2024},
}

@article{Dewage_2026,
 author = {Dewage, V.},
 title = {Toeplitz algebra of bounded symmetric domains: A quantum harmonic analysis approach via localization},
 fjournal = {Journal of Functional Analysis},
 journal = {J. Funct. Anal.},
 volume = {290},
 number = {2},
 pages = {111211},
 year = {2026},
}

@article{Sako,
 author = {Sako, H.},
 title = {{Property A and the operator norm localization property for discrete metric spaces}},
 fjournal = {Journal für die Reine und Angewandte Mathematik},
 journal = {J. Reine Angew. Math.},
 volume = {690},
 pages = {207-216},
 year = {2014},
}

@article{Georgescu2011,
 author = {Georgescu, V.},
 title = {On the structure of the essential spectrum of elliptic operators on metric spaces},
 journal = {J. Funct. Anal.},
 volume = {260},
 number = {6},
 pages = {1734--1765},
 year = {2011},
}

@article{GeorgescuIftimovici2006,
 author = {Georgescu, V. and Iftimovici, A.},
 title = {Localizations at infinity and essential spectrum of quantum Hamiltonians: I. General theory},
 journal = {Rev. Math. Phys.},
 volume = {18},
 number = {4},
 pages = {417--483},
 year = {2006},
}

\vspace{1cm}

\newpage
\begin{multicols}{2}

\noindent
Robert Fulsche\\
\href{fulsche@math.uni-hannover.de}{\Letter ~fulsche@math.uni-hannover.de}
\\
\noindent
Institut f\"{u}r Analysis\\
Leibniz Universit\"at Hannover\\
Welfengarten 1\\
30167 Hannover\\
GERMANY\\

\noindent
Raffael Hagger\\
\href{hagger@math.uni-kiel.de}{\Letter ~hagger@math.uni-kiel.de}
\\
\noindent
Mathematisches Seminar\\
Christian-Albrechts-Universität zu Kiel\\
Heinrich-Hecht-Platz 6\\
24118 Kiel\\
GERMANY

\end{multicols}
\end{document}